\documentclass[11pt,reqno,tbtags,a4paper]{amsart}
\usepackage{amssymb}
\usepackage{graphicx}
\usepackage{url}
\usepackage[square,numbers]{natbib}
\bibpunct[, ]{[}{]}{;}{n}{,}{,}

\title%[Random graphs with given vertex degrees]
{Random graphs with given vertex degrees and switchings}

\date{28 January, 2019; revised 30 January, 2019}
\author{Svante Janson}
\thanks{Partly supported by the Knut and Alice Wallenberg Foundation}
\address{Department of Mathematics, Uppsala University, PO Box 480,
SE-751~06 Uppsala, Sweden}
\email{svante.janson@math.uu.se}
\urladdr{http://www.math.uu.se/svante-janson}
\numberwithin{equation}{section}

\renewcommand\le{\leqslant}
\renewcommand\ge{\geqslant}

\allowdisplaybreaks

\theoremstyle{plain}% default
\newtheorem{theorem}{Theorem}[section]
\newtheorem{lemma}[theorem]{Lemma}

\newtheorem{corollary}[theorem]{Corollary}

\newtheorem{claim}{Claim}

\theoremstyle{definition}
\newtheorem{example}[theorem]{Example}

\newtheorem{remark}[theorem]{Remark}

\theoremstyle{remark}

\newenvironment{romenumerate}[1][-10pt]{% optional argument changes indentation
\addtolength{\leftmargini}{#1}\begin{enumerate}% gives (i), (ii) etc.
 }{\end{enumerate}}

\newenvironment{PXenumerate}[1]{%  argument yields prefix
\addtolength{\leftmargini}{-10pt}%
\begin{enumerate}% gives (#1 1), (#1 2) etc.
 }{\end{enumerate}}

\newcounter{oldenumi}
\newenvironment{PXenumerateq}[1]% continues numbering from previous romenumerate
{\setcounter{oldenumi}{\value{enumi}}%
\begin{PXenumerate}{#1}\setcounter{enumi}{\value{oldenumi}}}%
{\end{PXenumerate}}

\newcounter{thmenumerate}
\newenvironment{thmenumerate}
{\setcounter{thmenumerate}{0}%
 \def\item{\par% \ifnum\thethmenumerate=0\else\par\fi %\noindent\fi
 \refstepcounter{thmenumerate}\textup{(\roman{thmenumerate})\enspace}}
}
{}

\newcounter{xenumerate}   %no left indentation; thus wider lines

\newcommand\pfitem[1]{\par(#1):}
\newcommand\pfitemx[1]{\par#1:}
\newcommand\pfitemref[1]{\pfitemx{\ref{#1}}}

\newcounter{kasus}

\newcommand{\refT}[1]{Theorem~\ref{#1}}
\newcommand{\refTs}[1]{Theorems~\ref{#1}}

\newcommand{\refC}[1]{Corollary~\ref{#1}}
\newcommand{\refCs}[1]{Corollaries~\ref{#1}}
\newcommand{\refL}[1]{Lemma~\ref{#1}}
\newcommand{\refLs}[1]{Lemmas~\ref{#1}}
\newcommand{\refR}[1]{Remark~\ref{#1}}

\newcommand{\refS}[1]{Section~\ref{#1}}
\newcommand{\refSs}[1]{Sections~\ref{#1}}

\newcommand{\refE}[1]{Example~\ref{#1}}

\newcommand\REM[1]{{\raggedright\texttt{[#1]}\par\marginal{XXX}}}

\begingroup
  \divide\count255 by 60
  \multiply\count255 by -60
  \advance\count255 by \time
  \ifnum \count255 < 10 \xdef\klockan{\the\count1.0\the\count255}
  \else\xdef\klockan{\the\count1.\the\count255}\fi
\endgroup

\newcommand{\sumin}{\sum_{i=1}^n}

\newcommand\set[1]{\ensuremath{\{#1\}}}

\newcommand\xpar[1]{(#1)}
\newcommand\bigpar[1]{\bigl(#1\bigr)}
\newcommand\Bigpar[1]{\Bigl(#1\Bigr)}
\newcommand\biggpar[1]{\biggl(#1\biggr)}
\newcommand\lrpar[1]{\left(#1\right)}
\newcommand\bigsqpar[1]{\bigl[#1\bigr]}

\newcommand\xcpar[1]{\{#1\}}

\newcommand\bigabs[1]{\bigl\lvert#1\bigr\rvert}
\newcommand\Bigabs[1]{\Bigl\lvert#1\Bigr\rvert}

\def\rompar(#1){\textup(#1\textup)}    % usage: \rompar(...)

\newcommand\parfrac[2]{\lrpar{\frac{#1}{#2}}}

\newcommand\Bigparfrac[2]{\Bigpar{\frac{#1}{#2}}}

\def\xexp(#1){e^{#1}}

\newcommand\floor[1]{\lfloor#1\rfloor}

\newcommand\setn{\set{1,\dots,n}}

\newcommand\ntoo{\ensuremath{{n\to\infty}}}

\newcommand\norm[1]{\lVert#1\rVert}

\newcommand\Bignorm[1]{\Bigl\lVert#1\Bigr\rVert}

\newcommand\punkt{.\spacefactor=1000}    % om problem!
    
\newcommand\ie{i.e\punkt}
\newcommand\eg{e.g\punkt}

\newcommand\cf{cf\punkt}
\newcommand{\as}{a.s\punkt}

\newcommand\whp{w.h.p\punkt}
\newcommand\whpx{w.h.p}

\newcommand{\tend}{\longrightarrow}
\newcommand\dto{\overset{\mathrm{d}}{\tend}}
\newcommand\pto{\overset{\mathrm{p}}{\tend}}

\newcommand\eqd{\overset{\mathrm{d}}{=}}

\newcommand\op{o_{\mathrm p}}
\newcommand\Op{O_{\mathrm p}}

\newcounter{CC}
\newcounter{cc}
\newcommand\E{\operatorname{\mathbb E{}}}
\renewcommand\P{\operatorname{\mathbb P{}}}

\newcommand\Var{\operatorname{Var}}
\newcommand\Cov{\operatorname{Cov}}

\newcommand\Po{\operatorname{Po}}

\newcommand\ga{\alpha}
\newcommand\gb{\beta}
\newcommand\gd{\delta}
\newcommand\gD{\Delta}

\newcommand\gam{\gamma}
\newcommand\gG{\Gamma}

\newcommand\gl{\lambda}

\newcommand\gss{\sigma^2}

\newcommand\eps{\varepsilon}

\newcommand\cA{\mathcal A}
\newcommand\cB{\mathcal B}
\newcommand\cC{\mathcal C}

\newcommand\cE{\mathcal E}

\newcommand\cG{\mathcal G}
\newcommand\cH{\mathcal H}

\newcommand\cM{\mathcal M}

\newcommand\cP{\mathcal P}

\newcommand\cR{{\mathcal R}}
\newcommand\cS{{\mathcal S}}

\newcommand\indic[1]{\boldsymbol1\xcpar{#1}}

\newcommand\qq{^{1/2}}
\newcommand\qqw{^{-1/2}}

\newcommand\oi{\ensuremath{[0,1]}}

\newcommand\dtv{d_{\mathrm{TV}}}

\newcommand\lhs{left-hand side}
\newcommand\rhs{right-hand side}

\newenvironment{Aenumerate}[1][-10pt]{% optional argument changes indentation
\addtolength{\leftmargini}{#1}\begin{enumerate}% gives (A1), (A2) etc.

 }{\end{enumerate}}

{\setcounter{oldenumi}{\value{enumi}}
\begin{Aenumerate} \setcounter{enumi}{\value{oldenumi}}}
{\end{Aenumerate}}

\newcommand\ddn{\ensuremath{\mathbf{d}_n}}
\newcommand\dd{\ensuremath{\mathbf{d}}}
\newcommand\ddnx{\ensuremath{(d_i)_1^n}}
\newcommand\gndd{\ensuremath{G(n,\dd)}}

\newcommand\ggndd{\ensuremath{G^*(n,\dd)}}

\newcommand\hgndd{\ensuremath{\widehat{G}(n,\dd)}}
\newcommand\hgindd{\ensuremath{\widehat{G}_1(n,\dd)}}
\newcommand\hgiindd{\ensuremath{\widehat{G}_2(n,\dd)}}

\newcommand\xG{G^*}

\newcommand\hG{\widehat G}

\newcommand\ZZ{Z^*}

\newcommand\sC{\mathsf C}
\newcommand\sP{\mathsf P}

\newcommand\sPa{\sP_1}
\newcommand\sPb{\sP_2}
\newcommand\sPc{\sP_3}
\newcommand\sPd{\sP_4}
\newcommand\Xb{X_2}
\newcommand\Xc{X_3}
\newcommand\Yb{Y_{\sP_2}}
\newcommand\Yc{Y_{\sP_3}}

\newcommand\dmax{d_{\text{\rm max}}}

  \newcommand\hX{\widehat X}
\newcommand\gax{\tau}

\newcommand\GGn{{\mathfrak G}_n}
\newcommand\CM{configuration model}
\newcommand\SCM{switched configuration model}
\newcommand\hg{\SCM}
\newcommand\xX{X^*}

\newcommand\maxdi{\max_i d_i}
\newcommand\setG{\set{G}}

\newcommand\normm[1]{\norm{#1}_{\cM(\GGn)}}
\newcommand\comp{^{\textsf c}}
\newcommand\bcS{\cS\comp}
\newcommand\bcG{\cG\comp}
\newcommand\pc{p\comp}
\newcommand\medge{$m$-edge}
\newcommand\xedge[1]{$#1$-edge}
\newcommand\lm{_{\ell,m}}
\newcommand\umu{\gl}
\newcommand\hmu{\widehat\umu}
\newcommand\hmugc{\widehat\umu_{\cG}\comp}
\newcommand\hmuss{\widehat\umu_{\cS,s}}
\newcommand\hmugs{\widehat\umu_{\cG,s}}
\newcommand\zetass{\zeta_{\cS,s}}
\newcommand\zetags{\zeta_{\cG,s}}
\newcommand\refPG{\ref{P23}--\ref{Pgap2}}
\newcommand\refPP{\ref{P23}--\ref{Plast}}
\newcommand\Gs{_{\cG,s}}
\newcommand\Ss{_{\cS,s}}
\newcommand\Zgs{Z\Gs}
\newcommand\Zss{Z\Ss}
\newcommand\Iiajb{I_{i,\ga,j,\gb}}
\newcommand\XXc{\widehat{X}_3}
\newcommand\CS{Cauchy--Schwarz}
\newcommand\CSineq{\CS{} inequality}

\begin{document}

% \begin{comment}  
% Some suggestions:
% 05 Combinatorics 
% 05C Graph theory [For applications of graphs, see 68R10, 90C35, 94C15]
% 05C05 Trees
% 05C07 Vertex degrees
% 05C35 Extremal problems [See also 90C35]
% 05C40 Connectivity
% 05C65 Hypergraphs
% 05C80 Random graphs
% 05C90 Applications
% 05C99 None of the above, but in this section 
% 
% 60 Probability theory and stochastic processes
% 60C Combinatorial probability
% 60C05 Combinatorial probability
% 
% 60F Limit theorems [See also 28Dxx, 60B12]
% 60F05 Central limit and other weak theorems
% 60F17 Functional limit theorems; invariance principles
% 
% \end{comment}

\begin{abstract}
  Random graphs with a given degree sequence are often constructed using the
  configuration model, which yields a random multigraph.
  We may adjust this multigraph by a sequence of switchings, eventually
  yielding a simple graph. We show that, assuming essentially a bounded
  second moment of the degree distribution, this construction with the
  simplest types of switchings yields a simple random graph with an almost
  uniform distribution, in the sense that the total variation distance is
  $o(1)$.  This construction can be used to transfer results on
  distributional convergence from the configuration model multigraph to the
  uniform random simple graph with the given vertex degrees.
  As examples, we give a few applications to asymptotic normality.
  We show also a weaker result yielding contiguity when the maximum degree
  is too large for the main theorem to hold.
\end{abstract}

\maketitle

\section{Introduction}\label{S:intro}

We consider random graphs with
vertex set $[n]:=\setn$ and a given degree
sequence $\dd=\bigpar{d_1,\dots,d_n}$.
In particular, we define
$\gndd$ to be the random (simple) graph with degree sequence $\dd$
chosen uniformly at random among all such graphs.
%(We assume tacitly that $\dd$ is such that some such graph exists.) 
%Thus, vertex  $i$ has by definition degree $d_i$.
We will consider asymptotic results as \ntoo, where the degree sequence
$\dd=\ddn=(d_i^{(n)})_1^n$ depends on $n$,
but usually we omit $n$ from the notation.
%and write just $\dd$ and $d_i$. %$d_i$ for $d_i^{(n)}$.

The standard methods to constuct a random graph
with a given degree sequence
begin with the \emph{configuration model},
which was introduced by \citet{Bollobas-config}.
(See \cite{BenderCanfield,Wormald81} for related models and arguments.)
As is well-known, this method yields a random multigraph, which we
denote by $\ggndd$, with the given degree sequence $\dd$;
see \refS{Sconfig}.
This random multigraph may contain loops and multiple edges; however,
in the present paper (as in  many others), we will
consider asymptotic results as \ntoo, where
$\dd=\ddn$  satisfies
(at least)
\begin{align}\label{D2}
  \sumin d_i = \Theta(n),
  &&& \sumin d_i^2 = O(n),
\end{align}
%(and perhaps further conditions),
and then (see \eg{} the proof of \refL{LSG})
the expected number of loops and multiple edges is $O(1)$, which
might seem
insignificant when $n$ is large.
(Recall that $\Theta(n)$ means a number in the interval $[cn,Cn]$ for some
constants $c,C>0$.) %, possibly excepting some small $n$.)

In fact, we are mainly interested in the more regular case where, as \ntoo,
\begin{align}\label{D2lim}
  \frac{1}{n}\sumin d_i\to\mu,&&&
                          \frac{1}{n}\sumin d_i^2\to\mu_2
\end{align}
for some $\mu,\mu_2\in(0,\infty)$.
Obviously, \eqref{D2lim} implies \eqref{D2}. Conversely, if \eqref{D2}
holds, then there is always a subsequence satisfying \eqref{D2lim}.
It follows, see \refS{Ssubsub} for details,
that for our purposes \eqref{D2} and \eqref{D2lim} are essentially equivalent.
We will thus use the more general \eqref{D2} in the theorems.

In some applications, the random multigraph $\ggndd$ may be at least as
good as the simple graph $\gndd$. For example, this may be the case in
an application %(e.g.\ epidemiology models)
where the random graph is intended to be an
approximation of an unknown graph ``in real life'';
then the multigraph model may be just as good as an approximation.
On the other hand,
if we, as is often the case, really want a random simple graph (\ie, no
loops or multiple edges), then there
are several ways to proceed.

The standard method, at least in the sparse case studied in the present paper,
is to condition $\ggndd$  on the event that it is a
simple graph; it is a fundamental fact of the configuration model
construction (implicit in \cite{Bollobas-config})
that this yields a random simple
graph  $\gndd$ with the uniform distribution over all graphs with the given
degree sequence. This method has been very successful in many cases.
In particular, under the condition \eqref{D2} on $\dd$,
\begin{equation}
  \label{liminf}
  \liminf_\ntoo \P\bigpar{\ggndd \text{ is simple}}>0,
\end{equation}
see \eg{} \cite{SJ195,SJ281}, and then any result on convergence in
probability for $\ggndd$ immediately transfers to $\gndd$.
(See also \citet{BollobasRiordan-old}, where this method is used,
with more complicated arguments, also in cases with
$\P\bigpar{\ggndd\text{ is simple}}\to0$.)
However, as is also well-known, results on convergence in distribution do
not transfer so easily, and further arguments are needed.
(See
\cite{SJ196}, %showing asymptotic normality of the size of the $k$-core 
%(using a rather complicated extra argument for $\gndd$)
\cite{Riordan-phase}, %showing
%asymptotic normality of the size of the giant component in the weakly
%supercritical case (using a simple extra argument for $\gndd$
%noting that the proof only uses
%local explorations involving $o(n)$ vertices).
\cite{SJ338} for examples where this has succeded, with more or less
complicated extra arguments.)

Another method to create a simple graph from $\ggndd$ is to erase all
loops and merge any set of parallel edges into a single edge. This creates a
simple random graph, but typically its degree sequence is not exactly the
given sequence $\dd$.
Nevertheless, this \emph{erased \CM}
may be as useful as $\ggndd$ in some applications.
This construction is studied in \citet{BrittonDeijfenML}
and  \citet[Section 3]{Hofstad},
but will not be considered further in the present paper where we insist on
the degree sequence being exactly $\dd$.

In the present paper, we consider a different method, where we also
adjust $\ggndd$ to
make it simple, but this time we keep the degree sequence $\dd$ exact by using
switchings instead of erasing.
More precisely, we process the loops and multiple edges in $\ggndd$
one by one. For each
such bad edge, we chose another edge, uniformly at random,
and switch the endpoints of these two edges, thus replacing them by another
pair of edges. See \refS{Shg} for details.
Assuming \eqref{D2}, this
typically gives a simple graph after a single pass through the bad
edges (\refT{Tnotbad}); if not, we repeat until a simple graph is obtained.
We denote the resulting graph by $\hgndd$ and call it the
\emph{\hg}.
The idea to use switchings in this context goes back to \citet{McKay}
for the closely related problem of counting simple graphs with a given
degree sequence
(assuming $\max d_i = O(n^{1/4})$, but not \eqref{D2}),
and was made explicit for generating $\gndd$
%(for degree sequences with maximum degree $O(n^{1/4})$)
by \citet{McKayWormald1990} % 3- and 4-switchings; rejection sampling
(using somewhat different switchings).
% again assuming $\max d_i = O(n^{1/4})$).
See the survey by \citet{Wormald99} for further uses of switchings.
Recent refinements of the method, extending it to larger classes of degree
sequences by employing more types of switchings,
are given in
\citet{GaoW2016,GaoW2017,GaoW2018}.
% GaoWormald  regular $d=o(n\qq)$.
We will not use these recent refinements (that have been developed to handle
also rather dense graphs); instead we focus on the simple case when
\eqref{D2} holds and only a few switchings are needed; we also use only the
oldest and simplest types of switchings, used already by \citet{McKay}
(called \emph{simple switchings} in \cite{Wormald99}).
Although the switching method for this case has been known for a long time,
it seems to have been somewhat neglected.
Our purpose is to show that the switching method is powerful also in this
case, and that it complements the conditioning method discussed above for
the purpose of proving asymptotic results for $\gndd$.

\begin{remark}\label{Rcondition}
  From the point of view of constructing a random simple graph
  with given degree sequence by simulation, the standard approach using
  conditioning means that we sample the multigraph $\ggndd$; if it
    happens to be simple, we accept it, and otherwise we discard it
    completely and start again, repeating until a simple graph is found.
(See \eg{} \cite{Wormald1984}.)    
The approach in the present paper is instead to keep most of the multigraph
even when it is not simple, and resample only a few edges.
The disadvantage is that the result $\hgndd$
is not perfectly uniformly random, but \refT{T1} below shows that is
a good approximation, and asymptotically correct.
The advantage is that $\hgndd$ typically does not differ much from
$\ggndd$, and thus we often can show estimates of the type \eqref{diff0} in
\refC{C2} below.
\end{remark}

\begin{remark}
  In \eg{} \cite{McKayWormald1990,GaoW2017,GaoW2018}, an exactly uniformly
  distributed simple graph (\ie, $\gndd$) is constructed by combining
  switchings with rejection sampling, meaning that we may,
with some carefully calculated probabilities,
  abort the construction  and restart. 
  (Cf.\ the conditioning method where, as discussed in \refR{Rcondition}, we
  restart as soon as anything is wrong, instead of trying to fix it by
  switchings.) 
  Our focus is not on actual concrete construction of instances of $\gndd$
  by   simulation,
but rather to have a method of construction that can be used theoretically
to study properties  of $\gndd$,
  and for our purposes the approximate
  uniformity given by \refT{T1} is good enough. (And better, since the
  method is simpler.)
\end{remark}

\begin{remark}
  Switchings have also recently been used 
  (in a different way)
by \citet{AthreyaY}  to prove
  asymptotic normality for statistics of $\ggndd$ (in a subcritical case)
  using martingale methods.
\end{remark}

\begin{remark}
  For convenience, we state the results for a sequence $\ddn$ of degree
  sequences where 
  $\ddn$ has length (number of vertices) $n$.
More generally, one might consider a subsequence, or other sequences of
degree sequences $\dd_j$ with lengths $n_j\to\infty$.
This will be used in the proofs, see \refS{Ssubsub}.
\end{remark}

The main results are stated in \refS{Smain}, and proved in
\refSs{Spf0}--\ref{SpfT2}.
A few applications are given in \refS{Sapp}.

\section{Notation and main results}\label{Smain}
\subsection{Some notation}\label{SSnot}

Unspecified limits are as \ntoo;
\whp{} (with high probability) means with probability tending to 1 as \ntoo.
$\dto$ and $\pto$ denote convergence in distribution and probability,
respectively.

If $X_n$ are random variables and $a_n$ are positive numbers, then
$X_n=\Op(a_n)$ means $\lim_{K\to\infty}\sup_n \P(|X_n|>Ka_n)=0$,
and $X_n=\op(a_n)$ means $\sup_n \P(|X_n|>\eps a_n)=0$ for every $\eps>0$;
thus $X_n=\op(a_n)\iff X_n/a_n\pto0$.

Given a degree sequence $\dd=\ddnx$, we let
\begin{align}\label{dmax}
  \dmax&:=\max_{1\le i\le n} d_i,
  \\
  \label{N}
  N&:=\sumin d_i.
\end{align}
Thus a graph with degree sequence $\dd$ has $n$ vertices and $N/2$ edges.
Note that \eqref{D2} implies $N=\Theta(n)$ and $\dmax=O(n\qq)$.

If $\cS$ %=(\cS,\cF)$
is a measurable space, then $\cM(\cS)$ is the Banach
space of finite signed measures on $\cS$, and $\cP(\cS)$ is the subset of
probability measures.
If $\umu,\nu\in\cP(\cS)$, then
their \emph{total variation distance} is defined by
\begin{align}\label{dtv1}
  \dtv(\umu,\nu):=\sup_{A\subseteq\cS} \bigabs{\umu(A)-\nu(A)}
  =\tfrac12\norm{\umu-\nu}_{\cM(\cS)}
\end{align}
(where we tacitly only consider measurable $A$).  %, \ie, $A\in\cF$).
If $X$ and $Y$ are random elements of $\cS$ with distributions $\umu$ and
$\nu$, we also write
\begin{align}\label{dtv2}
  \dtv(X,Y):=\dtv(\umu,\nu)
  =
  \sup_{A\subseteq\cS} \bigabs{\P(X\in A)-\P(Y\in A)}.
\end{align}
If $\cS$ is \eg{} a separable metric space
(for example, as in our applications, a discrete finite or
countable set), then
\begin{align}\label{dtv3}
  \dtv(X,Y) = \min \P\bigpar{X'\neq Y'},
\end{align}
taking the minimum over all \emph{couplings} $(X',Y')$ of $X$ and $Y$, \ie,
pairs of random variables $X',Y'$ (defined on the same probaility space)
such that $X'\eqd X$ and $Y'\eqd Y$.
(See \eg{} \cite[Appendix A.1]{SJI} or \cite[Section 4]{SJ212}.)

If $\cS_n$, $n\ge1$,
is a sequence of measurable spaces, and $X_n$ and $Y_n$ are
random variables with values in $\cS_n$, then $X_n$ and $Y_n$ are
\emph{contiguous}
if for any sequence of measurable sets (events) $\cE_n\subseteq\cS_n$,
\begin{align}\label{contig}
  \P(X_n\in \cE_n)\to0
  \iff
  \P(Y_n\in \cE_n)\to0  .
\end{align}
See \eg{} \cite[Section 9.6]{JLR} and \cite{SJ212}.

If $G$ is a (multi)graph, we let %$V(G)$ denote its vertex set and
$E(G)$ denote its edge set
and $e(G):=|E(G)|$ its number of edges (counted with multiplicity).

$\sP_k$ denotes a path with $k$ edges and $k+1$ vertices,
and $\sC_k$ a cycle with $k$
vertices, $k\ge1$. In particular, $\sC_1$ is a loop, and $\sC_2$ is a pair
of parallel edges.
We denote the disjoint union of (unlabelled) graphs by $+$,
and write \eg{}
$2\sPb$ for $\sPb+\sPb$.

$C$ and $c$ denote positive constants that may be different at each occurrence.
(They typically depend on the sequence of degree sequences, but they do not
depend on $n$.)

\subsection{Main results}

$\hgndd$ is, by construction, a random simple graph with the given
degree sequence $\dd$.
However, it does not have a uniform distribution over all such
graphs, \ie, it will not be equal to the desired random graph $\gndd$; see
\refE{Edifferent}.
Nevertheless, our main result is the following theorem, which says in a
strong form that $\hgndd$ has asymptotically the same distribution as
$\gndd$; in the notation of \cite{SJ212},
$\hgndd$ and $\gndd$ are \emph{asymptotically equivalent}.
Hence, $\hgndd$ is a useful approximation of $\gndd$, and as stated
formally in \refC{C1} below,
results on both convergence in probability and convergence in distribution
that can be proved for $\hgndd$ transfer to $\gndd$.
Proofs are given in \refS{SpfT1}.

\begin{theorem}\label{T1}
  Assume that\/
  $\dd=(d_i^{(n)})_1^n$ depends on $n$ and satisfies the conditions
  \eqref{D2} and
  \begin{align}\label{dmaxo}
     \dmax=o\bigpar{n\qq}.
%      \max_i d_i^{(n)}=o\bigpar{n\qq}.
  \end{align}
  Then, as \ntoo,
  \begin{align}\label{t1a}
    \dtv\bigpar{\hgndd,\gndd}\to0.
  \end{align}
  In other words, there exists a coupling of $\hgndd$ and $\gndd$ such that
  \begin{align}\label{t1b}
    \P\bigpar{\hgndd\neq\gndd}\to0.
  \end{align}
\end{theorem}

\begin{corollary}\label{C1}
  Assume that\/ $\dd$ satisfies
    \eqref{D2} and \eqref{dmaxo}.
  Suppose that\/ $X_n=f_n\bigpar{\gndd}$ for some function $f_n$ of labelled
  simple graphs, and let
  $\hX_n=f_n\bigpar{\hgndd}$.
 If\/ $\gax$ is a constant such that $\hX_n\pto \gax$ as \ntoo, then also
 $X_n\pto \gax$. 
 % Similarly,
 More generally,
 if\/ $Y$ is a random variable such that $\hX_n\dto Y$ as \ntoo,
  then also $X_n\dto Y$.
\end{corollary}

Moreover, $\hgndd$ is obtained from $\ggndd$ using only a few
switchings. Hence it is often easy to prove the estimate \eqref{diff0}
below,
and then the next corollary shows that results on convergence in
distribution for $\ggndd$ transfer to $\hgndd$, using $\hgndd$ as an
intermediary in the proof.

\begin{corollary}\label{C2}
  Assume that\/ $\dd$ satisfies
    \eqref{D2} and \eqref{dmaxo}.
    Suppose that\/ $X_n=f_n\bigpar{\gndd}$ for some function $f_n$,
which is defined  more generally for labelled
  multigraphs, and let
  $\xX_n=f_n\bigpar{\ggndd}$.
  Suppose also that
  \begin{align}
    \label{diff0}
    f_n\bigpar{\hgndd}-
    f_n\bigpar{\ggndd}\pto0.
  \end{align}
% If\/ $\gax$ is a constant such that $\xX_n\pto \gax$ as \ntoo, then also
% $X_n\pto \gax$. 
% Similarly,
  If\/ $Y$ is a random variable such that $\xX_n\dto Y$ as \ntoo,
  then also $X_n\dto Y$.
\end{corollary}

We show in \refE{EO} that the condition $\maxdi=o(n\qq)$ is needed in
\refT{T1} and its corollaries above. However, we will also show the following
weaker statement without this assumption.
The proof is given in \refS{SpfT2}.

\begin{theorem}\label{T2}
  Assume that\/
  $\dd=(d_i^{(n)})_1^n$ depends on $n$ and satisfies 
  \eqref{D2}.
  Then, as \ntoo, the random graphs $\hgndd$ and $\gndd$ are contiguous.
  In other words, any sequence of events $\cE_n$ that holds \whp{} for
  $\hgndd$ holds also \whp{} for $\gndd$, and conversely.
\end{theorem}

\section{The construction of $\hgndd$}\label{Sconstr}

\subsection{The configuration model}\label{Sconfig}
The 
well-known \emph{configuration model}
was introduced by \citet{Bollobas-config} to generate
a random multigraph  with a given degree sequence $\dd=\ddnx$.
($N:=\sum_i d_i$ is assumed to be even.)
%(See \cite{BenderCanfield,Wormald81} for related models and arguments.)
The construction works by assigning a set of $d_i$ \emph{half-edges} to
each vertex $i$; this gives a total of $N$ half-edges.
A perfect matching of the half-edges is called a \emph{configuration}, and
defines a multigraph in the obvious way:
each pair of half-edges in the
matching is regarded as an edge in the multigraph.
We say that the configuration \emph{projects} to a multigraph.
We choose a configuration uniformly at random, and let $\ggndd$ be the
corresponding multigraph.

We denote the half-edges at a vertex $i$ by $i_1,\dots,i_{d_i}$.
%An edge $ij$ created by a pair of half-edges $\set{i_\ga,j_\gb}$ is denoted
%$i_\ga j_\gb$. (This gives a unique labelling of all edges, even when there
%are parallel edges.) 

Note that the mapping from configurations to multigraphs is not injective,
since we may permute the half-edges at each vertex without changing the
multigraph.
Nevertheless, we often informally identify a configuration and
the corresponding multigraph, and we use graph theory language for
configurations too.
In particular, a pair $\set{i_\ga,j_\gb}$ in
a configuration $\gG$ is called an edge in $\gG$, with endpoints $i$ and $j$,
and may be written $i_\ga j_\gb$;
similarly, the particular case
$\set{i_\ga,i_\gb}$ is called a loop,  two pairs (edges)
$\set{i_\ga,j_\gb}$ and $\set{i_\gam,j_\gd}$ are said to be parallel;
a configuration is simple if it has no loops or parallel edges.

\subsection{The \SCM}\label{Shg}

We construct the \hg{}
by first constructing a random configuration $\gG_0$ and the
corresponding multigraph $\ggndd$ as above.
Formally, we will do the switchings in the configuration, where all edges
are uniquely labelled;
they  induce corresponding switchings in the multigraph, and
informally we may think of the multigraph only.

We say that an edge in a configuration or multigraph
is \emph{bad} if it is a loop or if it is
parallel to another edge. If there is no bad edge in $\gG_0$
(or equivalently, in $\ggndd$),
then $\ggndd$ is simple,
and we accept it as it is.
Otherwise, we choose a bad edge in $\gG_0$,
say $i_\ga j_\gb$ (where $i=j$ in the case
of a loop), and choose another edge $k_\gam \ell_\gd$ uniformly at random
among all other edges in $\gG_0$;
we also order the two half-edges $k_\gam$ and
$\ell_\gd$ randomly. We then make a \emph{switching}, and replace the two edges
$i_\ga j_\gb$ and $k_\gam \ell_\gd$ by the new edges
$i_\ga \ell_\gd$ and $k_\gam j_\gb$.
This gives a new configuration $\gG_1$ on the same set of half-edges, and thus a
new multigraph $G_1$ that still has the same degree
sequence $\dd$. Moreover, we have removed one bad edge (in the case of
parallel edges, also another bad edge may have become good); however, it is
possible that we have created  a new bad edge (or several).
If the new configuration $\gG_1$ has no bad edge, then the corresponding
multigraph $G_1$ is simple and we stop; otherwise we pick a 
bad edge in $\gG_1$,  and repeat until we obtain a simple graph.
$\hgndd$ is defined to be the simple random graph we have when we terminate.

\begin{remark}\label{Rrule}
  The description above is somewhat incomplete, since we have not specified
  which bad edge we switch, if there is more than one.
  We assume that we have some fixed rule for this, \eg{}  the
  lexicographically  first bad edge, or a random one;
  different rules may yield somewhat different final 
 distributions, and thus formally different random graphs $\hgndd$,
see \refE{Epick},
but  our results hold for any such rule.
See \refL{Lrule}.
\end{remark}

Of course, if the degree sequence $\dd$ is not graphic, \ie,
no simple graph with this degree sequence exists, then
the switching process will never terminate. We conjecture that
if the sequence $\dd$ is graphic, then
the switching process almost surely terminates,
but we leave this as an open problem, see \refR{R4ever}.
However, we show in \refS{Spf0} the following theorem, which
is enough for our purposes; it
shows that assuming \eqref{D2}, %under the assumptions of \refT{T1},
% \whp{} no new bad edges are created, and thus
the process \whp{} terminates very quickly. 

\begin{theorem}
  \label{Tnotbad}
  Assume \eqref{D2}.
  Then, during the construction of $\hgndd$,
  \whp{} no new bad edges are created and the process terminates after
  $\Op(1)$ switchings.
\end{theorem}

\begin{remark}\label{Rdisjoint}
  It is easily seen that if we switch a bad edge
  $i_\ga j_\gb$ with an edge $k_\gam \ell_\gd$ that has a vertex
  in common with $i_\ga j_\gb$,
\ie, $\set{i,j}\cap\set{k,\ell}\neq\emptyset$,
  then there will always be a new bad edge created.
  It is therefore reasonable to modify the construction by  choosing the edge
  $k_\gam\ell_\gd$ uniformly at random among all edges vertex-disjoint from
  the bad   edge $i_\ga j_\gb$. (Provided this is possible, which it is
  \eg{} if the maximum 
  degree is $<N/4$, as is the case for all large $n$ when \eqref{D2} holds.)

  \refT{Tnotbad} implies that assuming \eqref{D2}, \whp{} we never switch
  two edges with a common vertex; hence the modified version will \whp{}
  yield exactly the same result $\hgndd$, and consequently
  \refTs{T1} and \ref{T2} holds for the
  modified construction too.

  \refE{Edifferent} shows that the modified construction does not yield
  exactly the same distribution of $\hgndd$, nor the uniform distribution.
\end{remark}

\begin{remark}\label{R4ever}
  \refT{Tnotbad} shows that, under our conditions, the switching process
  \whp{}  terminates with a simple graph after a finite number of switchings.
  A different question is whether it always terminates, for a given $n$ and
a graphic degree sequence   $\dd$. 
First, it is easy to see that the process might loop and never terminate, even
  using the modification in \refR{Rdisjoint}, see \refE{E4ever};
  however, in that example at least, this has probability 0.
  Hence the right question is whether the process terminates \as{} (\ie, with
  probability 1).

  It can be shown, see  
  \citet{Sjostrand},
  that there always exists a sequence of switchings leading to a simple
  graph. It follows that if we choose the bad edge to switch at random,
  then the
  process  terminates \as{} with a simple graph. (Note that the switching
  process is a finite-state Markov process, where the simple graphs are
  absorbing states.)
  We conjecture that the same holds for any rule choosing the bad edge to
  switch, but this remains an open problem.

  For completeness, if the switching process does not terminate,
  we define $\hgndd$ by restarting with a new random configuration.
  (This makes no difference for our results.)
\end{remark}

\subsection{Examples}\label{Sex}

\begin{example}\label{Edifferent}
  We consider a small example, both to illustrate the construction and to
  show that it does not yield perfect uniformity.
  
  Let $n=6$ and the degree sequence $\dd=(2,2,1,1,1,1)$. Thus $N=8$ and
  there are $N/2=4$ edges.
  There are $7!!=105$ different configurations of 5 different isomorphism
  types, as shown in Figure~\ref{fig:different}.
  Of these configurations, 72 yield simple graphs:
  48 of the type $\sPc + \sPa$ and 24 of the type $2\sPb$.
  The random simple graph $G(6,\dd)$ thus has the
  distribution
  \begin{align}\label{ex}
    \P\bigpar{G(6,\dd)\cong \sPc+\sPa} = \tfrac{2}{3},
    &&&
            \P\bigpar{G(6,\dd)\cong 2\sPb} = \tfrac{1}{3}.
  \end{align}
  The remaining 33 configurations yield non-simple graphs:
  24 $\sC_1+\sPb+\sPa$,
  6~$\sC_2+2\sPa$ and 3 $2\sC_1+2\sPa$.
  It is easily seen that modifying
  a graph $\sC_1+\sPb+\sPa$ by switching the bad edge and another one,
  randomly   chosen, gives
$\sPc+\sPa$ or $2\sPb$ with the same probabilities $\frac{2}3$ and
$\frac{1}3$ as in \eqref{ex}, while   $\sC_2+2\sPa$
may give $\sPc+\sPa$
but never $2\sPb$
(it may also give $2\sC_1+2\sPa$ or $\sC_2+2\sPa$ again if 
we switch the two parallel edges  with each other;
then further switchings are needed);
the final possibility $2\sC_1+2\sPa$
will give $\sC_1+\sPb+\sPa$ or $\sC_2+2\sPa$ after the first switching.
It follows that if we continue until we have a simple graph $\hgndd$, then
the probability of it being $\sPc+\sPa$ is strictly larger that $\frac{2}3$;
hence $\hgndd$ and $\gndd$ do not have the same distribution.
An elementary but uninteresting calculation shows that for this example,
  \begin{align}\label{hex}
    \P\bigpar{\hG(6,\dd)\cong \sPc+\sPa} = \tfrac{24}{35},
    &&&
            \P\bigpar{\hG(6,\dd)\cong 2\sPb} = \tfrac{11}{35}.
  \end{align}

The same holds also if we modify the construction by
always switching with an edge disjoint from the bad one
(see \refR{Rdisjoint}), although the exact probabilities will be different:
$\frac{31}{45}$ and $\frac{14}{45}$.

  \begin{figure}[ht]
    \centering
 \includegraphics[height=8cm]{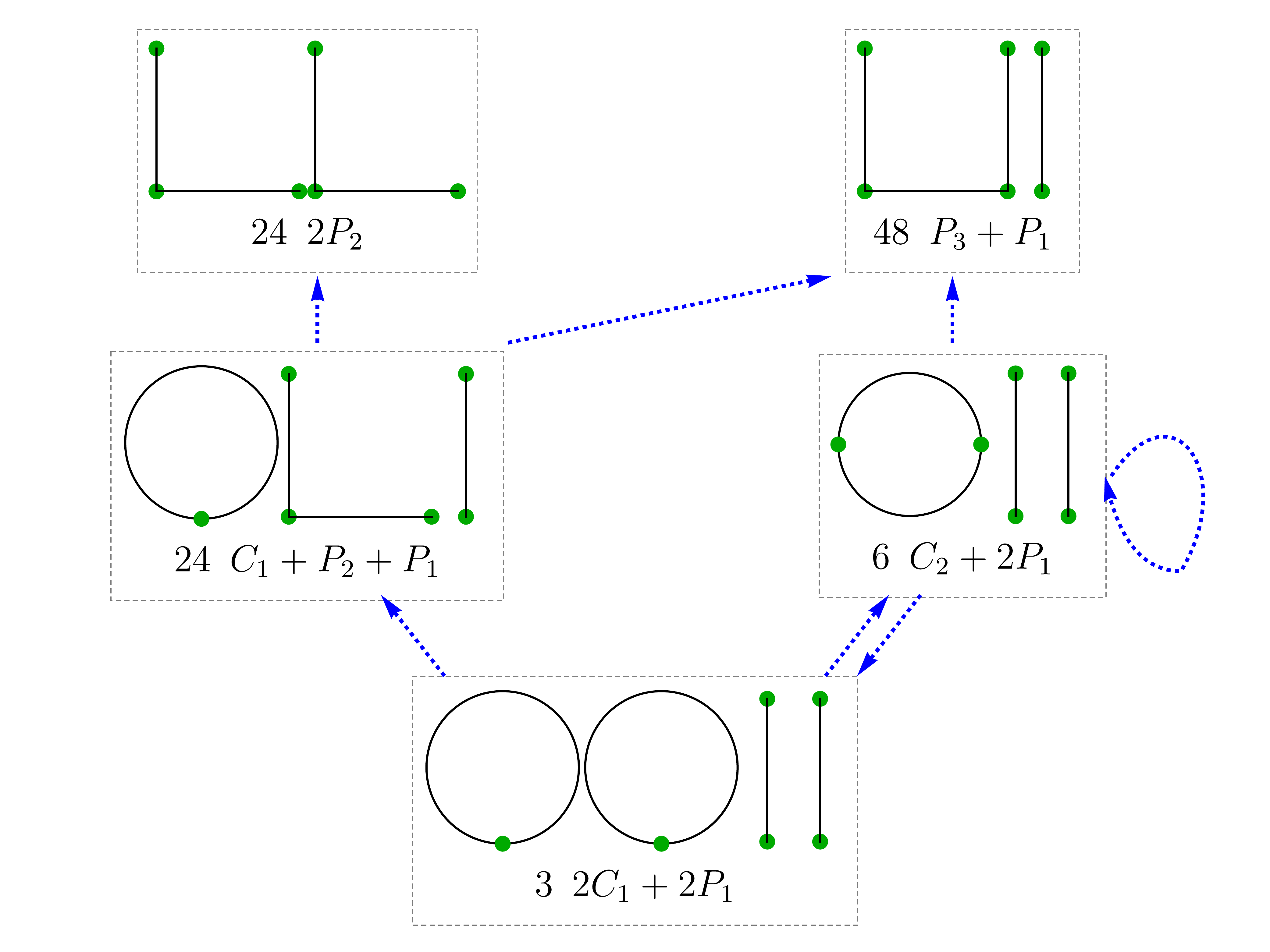}   
    \caption{Multigraphs (unlabelled) given by the degree sequence
      $(2,2,1,1,1,1)$ in \refE{Edifferent}.}
    \label{fig:different}
  \end{figure}
\end{example}

\begin{example}\label{Epick}
As another example, let $n=7$ and $\dd=(2,2,2,1,1,1,1)$; thus there are
$N/2=5$ edges. Suppose that a realization of the configuration model yields
the multigraph $\sC_1+\sC_2+ 2\sPa$ in Figure \ref{fig:pick}. There are
three bad edges: one loop and two parallel edges.

    \begin{figure}[ht]
    \centering
 \includegraphics[height=2cm]{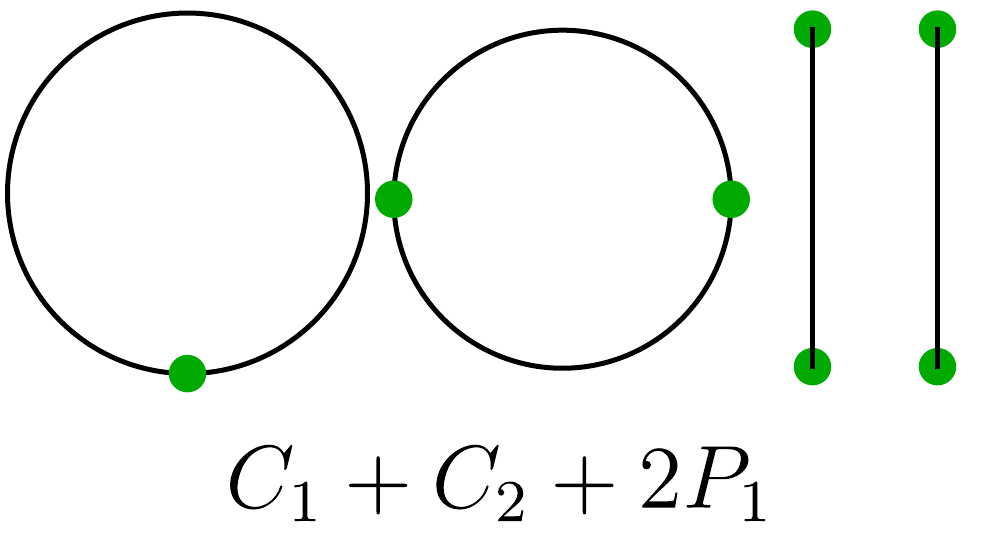}
    \caption{A multigraph given by the degree sequence
     $(2,2,2,1,1,1,1)$ in \refE{Epick}.}
    \label{fig:pick}
  \end{figure}

  If we choose to first switch the loop,
then with probability $\frac12$ we switch
it with one of the parallel edges, yielding the simple graph $\sC_3+2\sPa$.
If instead we switch the loop with one of the isolated edges, the result is
$\sC_2+\sPb+\sPa$, 
which after a second switching yields either a simple graph
$\sPd+\sPa$  or $\sPc+\sPb$,
or (if we switch the two parallel edges with each other) %
$\sC_2+\sPb+\sPa$ again or $2\sC_1+\sPb+\sPa$;
in the latter cases, we obtain
$\sPd+\sPa$  or $\sPc+\sPb$ after further switchings.
Hence,
the probability that the final graph contains a cycle
$\sC_3$ is $\frac12$.

On the other hand, if we begin by switching one of the parallel
edges, then there are three possibilities:
\begin{romenumerate}[-10pt]
  
\item\label{Epick1}
  With probability $\frac{1}4$, we switch with the loop, and
  obtain again $\sC_3+2\sPa$.
  
\item 
With probability $\frac12$ we switch with an isolated edge and
obtain $\sC_1+\sPc+\sPa$ and after a
second switching 
either $\sPd+\sPa$  or $\sPc+\sPb$.
\item 
%(with conditional probabilities $\frac{3}4$ and $\frac{1}{4}$,respectively);
With probability $\frac14$ we switch the two parallel edges with each other,
yielding either (with probability $\frac{1}{8}$ each) the 
same multigraph $\sC_1+\sC_2+2\sPa$ and we restart,
or $3\sC_1+2\sPa$.
In the latter case, we next switch a loop with either (probability $\frac12$
each)
another loop, yielding $\sC_1+\sC_2+2\sPa$ and we restart,
or with an isolated edge, yielding $2\sC_1+\sPb+\sPa$,
which eventually yields
either $\sPd+\sPa$  or $\sPc+\sPb$.
\end{romenumerate}
Summing up this case, we start again with a graph $\sC_1+\sC_2+2\sPa$ with
probability $\frac{3}{16}$, and thus
the total probability that we end with case \ref{Epick1} is
$\frac14/\frac{13}{16}=\frac{4}{13}$.

Consequently,
 conditioned on this realisation of $\xG(7,\dd)$,
the probability that the final graph $\hG(7,\dd)$ has a cycle $\sC_3$ is
$\frac12$ or $\frac4{13}$,
depending on our choice for the first switching. This
shows that the order of the switchings matters.
However, \refT{T1} is valid in any case, and thus such choices make no
difference asymptotically.

With the modification in \refR{Rdisjoint}, we never switch two parallel
edges with each other so some possibilities disappear in this example;
the final probabilities are
$\frac12$ and $\frac13$, but the conclusion remains the same.
\end{example}

\begin{example}\label{EO}
  Fix $a>0$, consider for simplicity
  only even $n$, and let $\dd:=(m,m,1,\dots,1)$, where
  $m:=\floor{\sqrt{an}}$. Thus all vertices except 1 and 2 have degree 1.
  Note that \eqref{D2} holds, but not \eqref{dmaxo}.
  
  Let $L_1$ and $L_2$ be the numbers of loops at 1 and 2, and
  let $M_{12}$ be the number  of edges 12 in the multigraph $\ggndd$.
 Note that besides these
  edges,  $\gndd$ contains $m-2L_j-M_{12}$ edges from $j$ to a leaf ($j=1,2$),
  and a perfect matching of all remaining vertices.
  In particular, there are   $n/2-O(n\qq)$ isolated edges.

  The multigraph $\ggndd$ is simple if $L_1=L_2=0$ and $M_{12}\le 1$.
  It is easy to see, \eg{}  by the method of moments,
  that asymptotically, $L_1\dto\Po(a/2)$, $L_2\dto\Po(a/2)$, 
  and $M_{12}\dto\Po(a)$, jointly with independent limits.
 When we construct $\hgndd$ by switchings, there is thus only
  $\Op(1)$ bad edges; moreover, \whp{} each switching will be with one of the
  $n/2-O(n\qq)$ isolated edges. In this case, no new bad edge is created by
  the switchings, 
  and we reach a simple graph $\hgndd$ after $L_1+L_2+(M_{12}-1)_+$ switchings;
  furthermore, no edge 12 is created by the switchings.
  It follows that \whp{} $\hgndd$ has an edge $12$ if and only if $M_{12}\ge1$.
  Consequently,
  \begin{align}\label{exbbh}
    \P\bigpar{12\in E(\hgndd)} =\P(M_{12}\ge 1)+o(1)\to \P\bigpar{\Po(a)\ge1}
    = 1-e^{-a}.
  \end{align}

  On the other hand, a simple graph with degree sequence $\dd$
  has either
  \begin{romenumerate}
   \item
    no edge 12 and $m$ edges from each of 1 and 2 to leaves $k\ge3$, together
  with a perfect matching of the remaining $n-2m-2$ vertices.
   \item
    an edge 12 and $m-1$ edges from each of 1 and 2 to leaves $k\ge3$, together
    with a perfect matching of the remaining $n-2m$ vertices.
  \end{romenumerate}
  Let the numbers of graphs of these two types   be $N_0$ and $N_1$.
  Then
  \begin{align}
    N_0&=\binom{n-2}{m}\binom{n-2-m}{m}(n-2m-3)!!
    \\
        N_1&=\binom{n-2}{m-1}\binom{n-1-m}{m-1}(n-2m-1)!!
  \end{align}
  and a simple calculation yields
  \begin{align}
    \frac{N_1}{N_0}=\frac{m^2}{n-2m}\to a.
  \end{align}
  Hence,
  \begin{align}\label{exbb}
    \P\bigpar{12\in E(\gndd)} = \frac{N_1}{N_0+N_1}\to \frac{a}{1+a}.
  \end{align}

Comparing \eqref{exbbh} and \eqref{exbb}, we see that the limits differ, and
thus \refT{T1} does not hold for this example.
Similarly, \refCs{C1} and \ref{C2} fail, for example if $f_n(G)$ is the
indicator of the event that the multigraph $G$ contains an edge
where both endpoints have degrees $\ge2$. 
This example shows that the condition \eqref{dmaxo} cannot be omitted from
\refT{T1} and its corollaries.
\end{example}

\begin{example}\label{E4ever}
      \begin{figure}[ht]
    \centering
 \includegraphics[height=14mm]{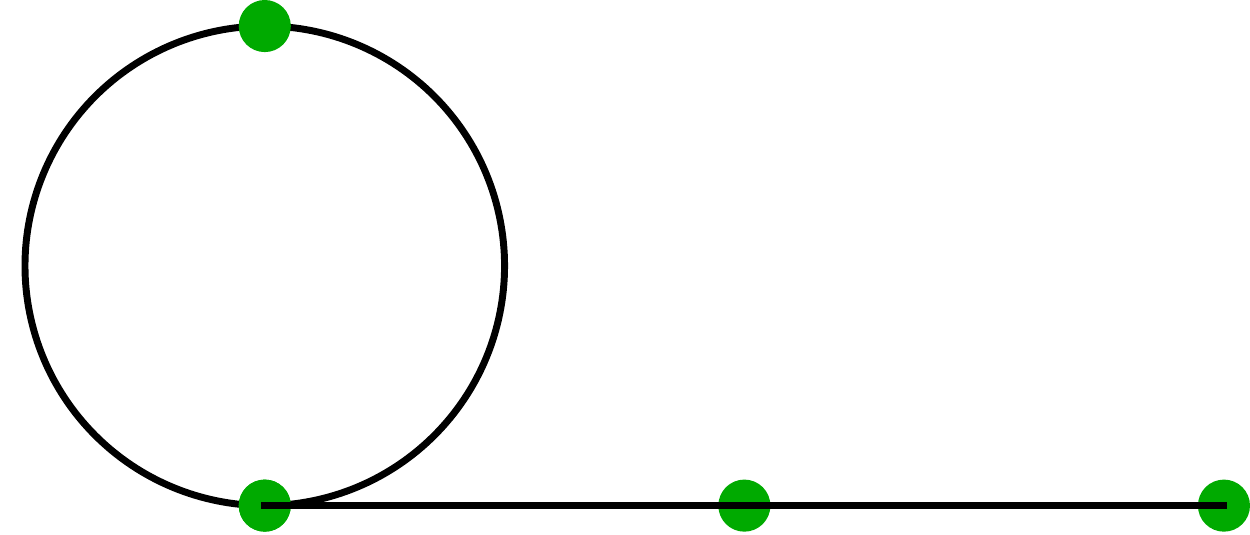}   
    \caption{A multigraph given by the degree sequence
      $(3,2,1,1)$ in \refE{E4ever}.}
    \label{fig:4ever}
  \end{figure}
  Let $\dd=(3,2,2,1)$, and suppose that the initial multigraph
  $\xG(4,\dd)$ has edges 12, 12, 13, 34, see Figure \ref{fig:4ever}.
  If we switch one of the parallel
  edges 12 with 34, then we may get a simple graph, but we may also create
  another edge 13 and get the edge set 12, 13, 13, 24. The latter multigraph
  is isomorphic to the original one, and we may continue and cycle between
  these two multigraphs for ever. Hence, there is no deterministic guarantee
  that the switching process always leads to a simple graph.
Note also the modification in \refR{Rdisjoint} does not help; we still can
make the same switchings.
  
  However, note that in this example, switching the same edges but in
  different orientations yields a simple graph.
  Since we order the half-edges at random when  switching, the infinite
  sequence of switchings above has probability 0; more generally, it is easily
  verified that for this example, \as{} the process terminates
  after a finite number of switchings.
\end{example}

\section{The distribution of $\hgndd$}\label{Spf0}

\subsection{More notation}\label{SMnot}
Let $S$ be number of switchings used in the construction.
Let %, as in \refS{Shg}, $\gG_0$ be the initial configuration given by the
%configuration model, and let
$\gG_k$ ($0\le k\le S$)
be the configuration after $k$ switchings, and let $G_k$ be the corresponding
multigraph. Thus $\ggndd=G_0$ and $\hgndd=G_S$.

Let $\cB_k$ be the set of endpoints of bad edges in $G_k$, and let $\cA_k$
be the set of their neighbours in $G_k$.

Let $b_k$ be the bad edge in $\gG_{k-1}$ chosen for the $k$th switching, and
let $e_k$ be the (random) other edge used in that switching.

An \emph{\medge} (in a graph or configuration) is a set of $m$ parallel edges
that are not loops, and such that there are no further edges parallel to
them.
(I.e., the multiplicity of the edge equals $m$.)
%$m$ parallel edges are said to form an \mple{} edge in $\gG$, and so on.
Let $L$ be the number of loops in $\gG_0$ (\ie, in $G_0=\ggndd$), and let
$M_m$ ($m\ge2$) be the number of \medge{s}.
%(We do not regard multiple loops as multiple edges.)
Furthermore, let
\begin{align}\label{MM}
  M:=\sum_{m\ge2} \binom m2 M_m,
\end{align}
the number of pairs of parallel edges in $\gG_0$.

Let $\GGn$ be the set of all simple graphs on $[n]$ with degree sequence
$\dd$.
Let
$\umu\in\cP(\GGn)$ be the distribution of $\gndd$, \ie, the uniform
distribution on $\GGn$, and let
$\hmu\in\cP(\GGn)$ be the distribution of $\hgndd$.

We sometimes tacitly assume that $n$ (and thus $N$) is large enough to avoid
trivialities (such as division by 0).

\subsection{Silver and golden}\label{ASG}

We say that %(a realization of)
the construction of $\hgndd$ is
\emph{silver} if
\begin{PXenumerate}{S}
  
\item \label{S1}
  No new bad edge is created during the construction.
\item\label{S2}
  No additional edge $e_k$ used for a switching has an endpoint in $\cB_0$.
%  No edge created by a switching is used (as $e_k$) in a later switching.
\end{PXenumerate}
The construction is \emph{golden} if it is silver and furthermore
\begin{PXenumerate}{G}
  \item\label{G1}
  $\ggndd$ has no triple edges. I.e., $M_m=0$ for $m\ge3$.
  
\item\label{G2}
  The loops and double edges in $\ggndd$ are vertex-disjoint.

\item\label{G3}
  The additional edges $e_k$ used for the switchings are vertex-disjoint
  with each other. % and   with the loops and double edges.
\end{PXenumerate}
Let $\cS$ and $\cG$ be the events that the construction is silver or golden,
respectively, and let $\bcS$ and $\bcG$ be their complements.
Furthermore, let $\cS_s:=\cS\cap\set{S=s}$
and $\cG_s:=\cG\cap\set{S=s}$ ($s\ge0$).

In a silver construction, each \medge{} is reduced to a single (good)
edge by $m-1$ switchings, and thus
\begin{align}\label{SS}
  S=L+\sum_{m\ge2} (m-1)M_m.
\end{align}
In particular, in a golden construction, recalling \eqref{MM}, 
\begin{align}\label{SG}
  S=L+M_2=L+M.
\end{align}
In a silver construction, \eqref{SS} and \eqref{MM} yield the inequality
\begin{align}
  \label{SS2}
  S\le L+M.
\end{align}

\begin{lemma}\label{LSG}
  \begin{thmenumerate}
  \item\label{LSG:S}
    If \eqref{D2} holds, then the construction of $\hgndd$ is \whp{} silver.
  \item\label{LSG:G}
    If \eqref{D2} and \eqref{dmaxo} hold,
    then the construction of $\hgndd$ is \whp{} golden.
  \end{thmenumerate}
\end{lemma}

\begin{proof}
  \pfitemref{LSG:S}
  A well-known simple calculation 
  shows that, assuming \eqref{D2},
  \begin{align}
    \E L& = \sum_i \frac{d_i(d_i-1)}{2(N-1)}
          \le \frac{1}{N}\sumin d_i(d_i-1)
                    \le \frac{1}{N}\sumin d_i^2
          = O(1),
\label{EL}    \\
    \E M& = \sum_{i<j} \frac{d_i(d_i-1)d_j(d_j-1)}{2(N-1)(N-3)}
          \le C \biggpar{\frac{1}{N}\sumin d_i(d_i-1)}^2 = O(1).
          \label{EM}
  \end{align}
  Hence, $\E(L+M)=O(1)$, \ie, there is a constant $C$ such that
  \begin{align}\label{ELMC}
    \E (L+M)\le C.
  \end{align}
  
  Fix a large integer $K$, and assume that $L+M\le K$.
  Then $|\cB_0|\le L+2M\le
  2K$. %Colour the vertices in $\cB$ black.
Let $k\ge0$, and
%  consider the $k$th switching, and
  suppose that the construction
  has been silver for the first $k$ switchings (in the obvious sense).
  Thus no new bad edges have been created and hence $\cB_{k}\subseteq\cB_0$.
  By the \CSineq{} and \eqref{D2},
  the number of half-edges belonging to vertices in $\cB_0$ is

\begin{align}\label{bk1}
  % |\cA_{k}| \le
  \sum_{i\in\cB_{0}} d_i
  \le \Bigpar{|\cB_{0}|\sum_{i\in\cB_{0}}d_i^2}\qq
  \le (2K)\qq\Bigpar{\sumin d_i^2}\qq  = O\bigpar{ n\qq},
\end{align}
and thus
\begin{align}
  \label{bka}
  |\cA_{k}|
  \le \sum_{i\in\cB_{k}} d_i
  \le \sum_{i\in\cB_{0}} d_i
  =O\bigpar{n\qq}.
\end{align}
Furthermore, the number of half-edges belonging to vertices in $\cA_k$ is,
by the argument in \eqref{bk1} together with \eqref{bka},
\begin{align}\label{bk2}
   \sum_{i\in\cA_{k}} d_i \le \Bigpar{|\cA_{k}|\sumin d_i^2}\qq
  = O\bigpar{|\cA_k|\qq n\qq}
  = O\bigpar{ n^{3/4}}.
\end{align}
It follows from \eqref{bk1} and \eqref{bk2} that when we pick a random edge
$e_{k+1}$
for the next switching, the probability that it has an endpoint in $\cB_0$ or
$\cA_k$ is $O(n^{-1/4})=o(1)$. Hence, \whp{} we switch with an edge
$e_{k+1}$ not
having any endpoint in $\cB_k\cup\cA_k$, and it is easy to see that then no
new bad edge is created.
Furthermore, \whp{} $e_{k+1}$ has no endpoint in $\cB_0$.
%only $2k\le 2K$ edges have been created by previous switchings,
%and thus \whp{} none of them is chosen for $e_{k+1}$.
Consequently, \whp{} the construction remains silver for the $(k+1)$th
swithching too. Since only $L+M\le K$ switchings are needed,
it follows by induction that 
\whp{} the construction is silver until the end.

We have shown that for every fixed $K$,
$\P\bigpar{\bcS\cap\set{L+M\le K}}\to0$.
Hence, using also  Markov's inequality and \eqref{ELMC},
\begin{align}\label{tak}
  \P(\bcS)&
            \le \P\bigpar{\bcS\cap\set{L+M\le K}}+ \P\bigpar{L+M>K}
  \le o(1)+\frac{\E(L+M)}{K}
\notag  \\
  &  \le o(1)+\frac{C}{K}.
\end{align}
Thus $\limsup_\ntoo \P(\bcS)\le C/K$. Since $K$ is arbitrary, $\P(\bcS)\to0$.

\pfitemref{LSG:G}
The expected number of triples of parallel edges in $\ggndd$ is at most,
using \eqref{D2} and \eqref{dmaxo},
\begin{align}
  \sum_{i<j}d_i^3d_j^3\frac{C}{N^3}
  \le C \frac{\dmax^2}{N^3}\Bigpar{\sum_i d_i^2}^2
  \le C  \frac{\dmax^2}{n}
=o(1).
\end{align}
Hence, \ref{G1} holds \whpx.

Similar calculations show that
the expected number of pairs of 2 loops, a loop and a double edge,
or 2 double edges, are $o(1)$. Hence,  \ref{G2} holds \whpx.

Finally, fix $k\ge1$.
Given $e_1,\dots,e_{k-1}$, these have (at most) $2(k-1)$ endpoints.
There are at most $2(k-1)\dmax$ edges with an endpoint in this set.
Since $e_k$ is drawn at random among the $N/2-1$ edges distinct from $b_k$,
the probability that $e_k$ is not
vertex-disjoint from $e_1,\dots,e_{k-1}$ is at most
$2(k-1)\dmax/(N/2-1)=o(1)$.
Note also that if the construction is silver and $L+M\le K$, then at most
$K$ switchings are done by \eqref{SS2}.
It follows that for any fixed $K$,
\begin{align}
%    \P\bigpar{\bcG\cap\set{L+M\le K}}\le \P(\bcS)+
  \P\bigpar{\bcG\cap\cS\cap\set{L+M\le K}}=o(1).
\end{align}
The argument in (and after) \eqref{tak} shows that
$  \P\bigpar{\bcG\cap\cS}\to0$.
Hence, using also part \ref{LSG:S},
%\begin{align}
$  \P\bigpar{\bcG}\le \P\bigpar{\bcG\cap\cS}+\P\bigpar{\bcS}\to0$.
%\end{align}
\end{proof}

\begin{proof}[Proof of \refT{Tnotbad}]
  Immediate from \refL{LSG}\ref{LSG:S}, since a silver construction creates
  no new bad edges by \ref{S1} and uses $S\le L+M$ switchings by
  \eqref{SS2}, so $\E S\le \E(L+M)=O(1)$ by \eqref{ELMC} and thus
  $S=\Op(1)$.
\end{proof}

\subsection{The choice of a bad edge}\label{Schoice}

As said in \refR{Rrule},
the random graph $\hgndd$ may depend on the (unspecified) rule for choosing
the bad edge for each switching.
However, all rules yield asymptotically the same result, at least provided
\eqref{D2} holds.

\begin{lemma}
  \label{Lrule}
  Assume \eqref{D2}.
  Let $\hgindd$ and $\hgiindd$ be created by using two different rules
  for choosing the bad edge for each switching.
  Then
  $\dtv(\hgindd,\hgiindd)\to0$.
\end{lemma}

\begin{proof}
  Suppose that we have a silver construction
of $\hgindd$; then only edges that
are bad already in $\gG_0$ will be switched.
It follows that 
we may couple the two constructions of $\hgndd$,
starting with the same $\gG_0$, such that if the construction of $\hgindd$
is silver, then, in the sequence of graphs $G_0,\dots,G_S=\hgndd$,
exactly the same switchings are made in both constructions,
although perhaps in different
order; consequently, the 
two constructions yield the same $\hgndd$.
(On the level of configurations, the switchings may differ, because of the
choice of one parallel edge out of several.)
Consequently,
for this coupling, by \refL{LSG}\ref{LSG:S},
\begin{align}
  \P\bigpar{\hgindd\neq\hgiindd}=o(1).
\end{align}
This proves the lemma by \eqref{dtv3}.
\end{proof}

\refL{Lrule} implies that if \refT{T1} or \refT{T2} holds for some rule, then
it holds for any rule. We may thus for the proofs below assume that we
each time choose the bad edge in $\gG_k$
that is first according to the following order.
\begin{PXenumerate}{B}
  
\item\label{B1}
  First the loops, in  lexicographic order.
  
\item\label{B2}
  The \medge{s} in lexicographic order of their endpoints, and for
  each \medge{} its $m$ edges in $\gG_k$ in lexicographic
  order.
\end{PXenumerate}
The exact definition of the lexicographic order in these cases is left to
the reader. In fact, any fixed order would do.

\subsection{The subsubsequence principle}\label{Ssubsub}
Next we note that it suffices to prove that \refTs{T1} and \ref{T2}
always hold for some subsequence.
This is a general argument, which we repeat for convenience:
Suppose that \refT{T1} fails; then there exists a sequence $\ddn$ satisfying
the assumptions and such that \eqref{t1a} fails; thus there exists
$\eps>0$ and a subsequence
such that
$\dtv\bigpar{\hgndd,\gndd}\ge\eps$ for every $n$ in the subsequence.
But by assumption we can find a subsubsequence such that \refT{T1} holds, a
contradiction.  The proof for \refT{T2} is essentially the same.

In particular, assuming \eqref{D2},
by selecting a suitable subsequence we may in the remainder of the proofs
assume that \eqref{D2lim} holds for
some $\mu>0$ and $\mu_2<\infty$. Let $\nu:=\mu_2-\mu$; then \eqref{D2lim}
implies $N/n\to\mu$ and 
\begin{align}\label{nu}
  \frac{1}{n}  \sumin d_i(d_i-1)\to\nu,
  &&&
      \frac{1}{N}  \sumin d_i(d_i-1)\to\frac{\nu}{\mu}.
\end{align}

If $\nu=0$, then \eqref{nu} and
\eqref{EL}--\eqref{EM} show that $\E L\to0$ and $\E M\to0$.
Consequently, \whp{} $L=M=0$, so $\ggndd$ is simple, in which case
$\hgndd=\ggndd$; thus $\dtv(\hgndd,\ggndd)\to0$.
Furthermore, by the fact that $\gndd$ has the same distribution as
$\ggndd$ conditioned on being simple, we can couple $\gndd$ and $\ggndd$
such that they are equal when $\ggndd$ is simple; thus by \eqref{dtv3},
\begin{align}
  \dtv(\gndd,\ggndd)\le \P\bigpar{\ggndd\text{ is not simple}} \to0.
\end{align}
Hence \refTs{T1} and \ref{T2} follow trivially when $\nu=0$.
Consequently, in the proofs we may assume $\nu>0$.

\subsection{Silver constructions}
  %The (conditional) distribution of $\hgndd$}
\label{Ssilver}
Consider a silver construction.
Each switching of a loop creates a copy of $\sPb$, and each
set of $m-1$ switchings of an \medge{} creates $m-1$ copies of $\sPc$
having a common middle edge, which is the one edge remaining of the original $m$
parallel ones.
Colour the created copies of $\sPb$ and $\sPc$ % in the configuration
\emph{red}; these are regarded as (not
necessarily disjoint) subgraphs of the configuration.
Note that the non-leaves in the red paths belong to $\cB_0$,
while a leaf is an endpoint of some $e_k$ and thus, by \ref{S2}, lies
outside $\cB_0$. 
By \ref{S2}, the edges  in the 
red paths %copies of $\sPb$ and $\sPc$
will not be used
by later switchings, and thus the red paths
remain as subgraphs of $\gG_S$,  and thus of $G_S=\hgndd$.

The red paths do not have to be vertex-disjoint. However,
by the remarks above and \ref{S1}--\ref{S2},
the set $\cR$ of red paths in $\gG_S$, or equivalently in $G_S=\hgndd$,
has the following
properties.
\begin{PXenumerate}{P}
\item \label{P23}
  Each red path has length 2 or 3, \ie, is a copy of $\sPb$ or $\sPc$.
  \item \label{Pe}
  The red paths  are edge-disjoint, except that several red
  paths $\sPc$ may share the same middle edge.
\item\label{Pv}
  An leaf of a red path is not a non-leaf of another red path.
%If $\cD$ is the set of endpoints of the red paths, and $\cB$ is the set of 
%non-endpoints in them
\end{PXenumerate}

Define the \emph{gap} of a red path %$\sPb$ or $\sPc$
as its pair of endpoints.
This is the pair of endpoints of the edge $e_k$ used to create this red
path.
By \ref{S1} and \ref{S2}, $e_k$ was a good edge in $\gG_0$, \ie, there
was no parallel edge in $\gG_0$. Thus the red paths have also
the properties:
\begin{PXenumerateq}{P}
\item\label{Pgap1}
  The gaps of the red paths are distinct pairs of vertices.
\item\label{Pgap2}
  The gaps of the red paths are 
  non-edges in $G_S$. %\hgndd$. 
\end{PXenumerateq}

Furthermore, by \ref{B2}:

\begin{PXenumerateq}{P}
  \item\label{Plast}
If a red $\sPc$ is given by the edges
$i_\ga j_\gb$, $j_\gam k_\gd$, $k_\eps \ell_\zeta$ in $\gG_S$, then
necessarily
the edge $j_\gam k_\gd$ comes after $j_\gb k_\eps$ in the lexicographic order.
\end{PXenumerateq}

Conversely, in a silver construction,
the red paths in $\gG_S$ determine precisely the switchings that have
been made; hence they together with $\gG_S$ determine the initial
configuration $\gG_0$ and also, by \ref{B1}--\ref{B2},  the order of the
switchings.
Moreover, given any simple
configuration $\gG$ (on the given set of half-edges) with corresponding
graph $G$ and
a set of red paths in $\gG$ (or $G$)
satisfying \ref{P23}--\ref{Plast}
(with obvious notational changes here and below:
$\gG_S$ is replaced by $\gG$ and $G_S$ by $G$),
there exists a unique initial configuration $\gG_0$ and a unique silver
sequence of switchings, satisfying \ref{B1}--\ref{B2},
that yields $\gG_S=\gG$ with the given red paths.
Each such history with a given number $S$ of switchings has the same
probability. 

Consequently, 
dropping ``red'':
\begin{claim}
\label{ClgG}
For a fixed $s\ge0$,
the
conditional probability $\P\bigpar{\gG_S=\gG\mid\cS_s}$ is proportional to
the number of sets of $s$ paths in $\gG$ that satisfy \refPP.
\end{claim}

We project to graphs. For a simple graph $G$ and a configuration $\gG$
projecting to $G$, there is an obvious bijection between sets of paths in
$G$ and sets of paths in $\gG$. Note that the conditions
\refPG{} depend only on the paths in $G$,
while \ref{Plast} depends also on the specific configuration $\gG$.

Given a simple graph $G\in\GGn$, and a set $\cR$ of $s$
paths in $G$ that satisfy \refPG, let the \emph{weight}
$w(\cR;G)$ be the
probability that the lifting of $\cR$ to paths in
a configuration $\gG$, chosen uniformly at random among all
configuration $\gG$ that project to $G$,
satisfies also \ref{Plast}.
Furthermore, for $G\in\GGn$ and $s\ge0$, let
$\zetass(G)$ be the sum of the weights of all sets of $s$ paths in $G$ that
satisfy \refPG.

Recall that each simple graph $G\in\GGn$ is the projection of the same
number
$A:=\prod_i d_i!$
of configurations.
It follows that
%for a fixed $s\ge0$,
%and all $G\in\GGn$,
the number of pairs $(\gG,\cR)$ where $\gG$ is a configuration projecting to
$G$ and $\cR$ is a set of $s$ paths in $\gG$ satisfying \refPP{}
equals
$\sum_{\cR} w(\cR;G)A=A\zetass(G)$.
Consequently,  
Claim \ref{ClgG} implies:
\begin{claim}\label{ClG}
  For a fixed $s\ge0$,
  and all $G\in\GGn$,
the
conditional probability $\P\bigpar{\hgndd=G\mid\cS_s}$
is proportional to
$\zetass(G)$.
\end{claim}

Let $\hmuss$ be the distribution of $\hgndd$ conditioned on
$\cS_s$, \ie,
\begin{align}\label{murs}
  \hmuss\setG=\P\bigpar{\hgndd=G\mid\cS_{s}}.
\end{align}
Claim \ref{ClG} says that the probability $\hmuss(G)$
is  proportional to
$  \zetass(G)$,
and thus to $\zetass(G)\umu\setG$, recalling that $\umu$ is the uniform
distribution.
To find the normalizing constant, recall that
$\umu$ is the distribution of $\gndd$, which implies,
using the notation $\Zss:=\zetass\bigpar{\gndd}$,
\begin{align}
  \sum_{G\in\GGn}\zetass(G)\umu\setG
  =\E\bigsqpar{ \zetass\bigpar{\gndd}}
  =\E \Zss.
\end{align}
Consequently,
since $\hmuss$ is a probability measure, 
\begin{align}\label{bw}
  \hmuss\setG
  = \frac{\zetass(G)}{\E \Zss} \umu\setG.
\end{align}

\subsection{Golden constructions}\label{Sgolden}
For the proof of \refT{T1}, we simplify and consider only golden
constructions. In a golden construction, it follows from \ref{G1}--\ref{G3}
that the red paths are vertex-disjoint.
Conversely, a silver construction yielding vertex-disjoint red paths is golden.

It follows that Claims \ref{ClgG} and \ref{ClG} above hold also
if we replace $\cS_s$ by $\cG_s$ and consider only
sets $\cR$ of vertex-disjoint paths, so $\zetass(G)$ is replaced by
%for $G\in\GGn$ and $s\ge0$,
$\zetags(G)$, defined as the total weight of all sets of $s$ vertex-disjoint
paths in $G$ that satisfy
\refPG.
Thus, in analogy to \eqref{murs}--\eqref{bw},
letting $\hmugs$ be the distribution of $\hgndd$ conditioned on
$\cG_s$, and $\Zgs:=\zetags(\gndd)$,
\begin{align}\label{murg}
  \hmugs\setG:=\P\bigpar{\hgndd=G\mid\cG_{s}}
  = \frac{\zetags(G)}{\E \Zgs} \umu\setG.
\end{align}
Hence,
\begin{align}
  \normm{\hmugs-\umu}&
  =\sum_{G\in\GGn}\bigabs{\hmugs\setG-\umu\setG}
    =\sum_{G\in\GGn}\Bigabs{\frac{{\zetags(G)-\E \Zgs}}{\E \Zgs}}\umu\setG
       \notag\\&      =\frac{\E\bigabs{\Zgs-\E \Zgs}}{\E \Zgs}.
  \label{bob}
\end{align}
This will be studied in the following sections.
We first find the weights $w(\cR;G)$.

\begin{lemma}\label{Lw}
 Suppose that
$G\in\GGn$ and that
  $\cR$ consists of $\ell\ge0$ paths $\sPb$ and $m\ge0$ paths $\sPc$ in $G$,
  all vertex-disjoint.
  Then the weight $w(\cR;G)=2^{-m}$.
\end{lemma}

\begin{proof}
  Given any configuration projecting to $G$, we obtain all other such
  configurations by permuting the half-edges at each vertex.
  Hence, if $H\cong\sPc$ is a path in $G$, then it follows by symmetry,
  permuting only the half-edges at the two central vertices of $H$,
  that the probability is $\frac12$
  that \ref{Plast} holds for the lift of $H$ to a random configuration $\gG$ 
  that projects to $G$. Furthermore, for disjoint $H_1,\dots,H_m$, the
  corresponding events are independent.
  The result follows.
\end{proof}
%(If $H$ is the path $ijk\ell$, and $\gG$ has edges
%$i_\ga j_\gb$, $j_\gam k_\gd$, $k_\eps \ell_\zeta$, consider also the
%configuration where these are replaced by
%$i_\ga j_\gam$, $j_\gb k_\eps$, $k_\gd \ell_\zeta$;
%then \ref{Plast} holds for exactly one of the two configurations.) 

For $\ell,m\ge0$ and $G\in\GGn$, let $\zeta\lm(G)$ be the number of sets
$\set{F_1,\dots, F_\ell,\allowbreak H_1,\dots,H_m}$ of vertex-disjoint
paths in $G$ 
such that each $F_i\cong\sPb$ and each $H_j\cong\sPc$,
and \ref{Pgap2} holds.
(We ignore the order of $F_1,\dots,F_\ell$ and $H_1,\dots,H_m$.)
Note that  \ref{Pe}--\ref{Pgap1} holds for any set of vertex-disjoint paths.
Hence, \refL{Lw} yields
\begin{align}\label{alice}
  \zetags(G)=\sum_{\ell+m=s}2^{-m}\zeta\lm(G).
\end{align}

\begin{remark}
  The proof of \refL{Lw} is easily extended to the more general sets of
  paths in \refS{Ssilver}. If the paths $\sPc$ in $\cR$ are grouped
  according to their middle edges, with $m_k$ groups of $k-1$ paths having
  the same middle edge (and thus coming from switchings of an \xedge{k} in
  $\gG_0$), $k\ge2$,
  then $w(\cR;G)=\prod_k k^{-m_k}$.
  We will not use this formula, and omit the details.
\end{remark}

\section{Some subgraph counts in $\ggndd$}\label{Sgg}

The equations \eqref{bob} and \eqref{alice} show that it suffices to show
good estimates for the special subgraph counts
% $Z\lm:=\zeta\lm(\gndd)$.
$\zeta\lm(\gndd)$.
In order to do so, we use the standard method
to study the random multigraph $\ggndd$ instead
(\cf{} \refS{S:intro}).
Let 
$\ZZ\lm:=\zeta\lm\bigpar{\ggndd}$.

For two multigraphs $H$ and $G$, let $\psi_H(G)$ be the number of subgraphs
of $G$ isomorphic to $H$.
Define
\begin{align}\label{trox}
  X_H:=\psi_H\bigpar{\ggndd}.
\end{align}
Note  that $L=X_{\sC_1}$ and $M=X_{\sC_2}$.
We are mainly interested in the case when $H=\sPb$
or $\sPc$, and we write $X_k:=X_{\sP_k}$.

We begin with an estimate that does not require \eqref{dmaxo}.

\begin{lemma}
  \label{LA0}
  Assume that\/ $\dd$ satisfies  \eqref{D2}. % and \eqref{dmaxo}.
  Then, for every fixed $\ell,m\ge0$,
  \begin{align}
\E\bigpar{ \Xb^\ell \Xc^m} = O\bigpar{n^{\ell+m}}.
  \end{align}
\end{lemma}

\begin{proof}
  First, deterministically using \eqref{D2},
  since there are at most $\binom{d_i}2$ copies of
  $\sPb$ with middle vertex $i$, %in $\ggndd$
  \begin{align}\label{fru}
    \Xb\le \sumin \binom{d_i}2
    =O(n).
  \end{align}
%  (The inequality is an equality unless there is a loop.)
%Hence, it suffices to consider the case $\ell=0$.

We estimate $\Xc$ too from above by overcounting.
For $i,j\in[n]$ and $\ga\in[d_i]$, $\gb\in[d_j]$, let $\Iiajb$ be the
indicator of the event that the half-edges $i_\ga$ and $j_\gb$ form an edge.
Then
\begin{align}\label{XXc}
  \Xc\le\XXc:=\sum_{i<j}\sum_{\ga=1}^{d_i}\sum_{\gb=1}^{d_j} (d_i-1)(d_j-1)\Iiajb.
\end{align}
We show, by induction on $m$, that for every fixed $m\ge0$,
\begin{align}\label{glu}
  \E\XXc^m =O\bigpar{ n^m}.
\end{align}
This is trivial for $m=0$, and for $m=1$ we have (\cf{} the similar \eqref{EM})
\begin{align}\label{EXXc}
  \E\XXc
&  =  \sum_{i<j}\sum_{\ga=1}^{d_i}\sum_{\gb=1}^{d_j} (d_i-1)(d_j-1)\E\Iiajb
    =  \sum_{i<j}\frac{d_i(d_i-1)d_j(d_j-1)}{N-1}
           \notag\\&
  \le \frac{C}{n}\lrpar{ \sumin d_i(d_i-1)}^2
  =O(n).
\end{align}

For the induction step,  note that  \eqref{XXc} yields the expansion
\begin{align}\label{XXcm}
  \E\XXc^m=\sum_{i_1<j_1,\dots,i_m<j_m}\sum_{\ga_1,\dots,\ga_m}\sum_{\gb_1,\dots,\gb_m}
\E  \prod_{k=1}^m (d_{i_k}-1)(d_{j_k}-1) I_{i_k,\ga_k,j_k,\gb_k}.
\end{align}
First, let $1\le k<\ell\le m$, and consider all terms in \eqref{XXcm} where
$\xpar{i_k,\ga_k,j_k,\gb_k}=\xpar{i_\ell,\ga_\ell,j_\ell,\gb_\ell}$.
In these terms, $I_{i_\ell,\ga_\ell,j_\ell,\gb_\ell}$ is redundant, and
$(d_{i_\ell}-1)(d_{j_\ell}-1)<\dmax^2\le Cn$, and eliminating these factors
yields $\E\XXc^{m-1}$. Hence, the induction hypothesis shows that the
contribution of these terms is at most
$ Cn\E\XXc^{m-1}=O(n^m)$.

Summing over all pairs $(k,\ell)$ still yields $O(n^m)$.

The remaining terms in \eqref{XXcm} are those where the $m$ quadruples
$\xpar{i_k,\ga_k,j_k,\gb_k}$ are distinct. In this case,
either
\begin{align}
  \E \prod_{k=1}^m I_{i_k,\ga_k,j_k,\gb_k}
&  =\frac{1}{(N-1)(N-3)\dotsm(N-2m+1)}
  \\&
  \le \frac{C}{(N-1)^m}
  = C \prod_{k=1}^m \E I_{i_k,\ga_k,j_k,\gb_k},
\end{align}
(where $C=C_m$ depends on $m$),
or $\prod_{k=1}^m I_{i_k,\ga_k,j_k,\gb_k}=0$ identically
because of conflicts.
Consequently, the sum of these terms in \eqref{XXcm} is at most
\begin{align}\label{XXcm2}
\sum_{i_1<j_1,\dots,i_m<j_m}\sum_{\ga_1,\dots,\ga_m}\sum_{\gb_1,\dots,\gb_m}
C  \prod_{k=1}^m (d_{i_k}-1)(d_{j_k}-1) \E I_{i_k,\ga_k,j_k,\gb_k}
  = C \bigpar{\E\XXc}^m.
\end{align}
By \eqref{EXXc}, this too is $O(n^m)$, which completes the induction step
and proves \eqref{glu}.
The result follows by \eqref{fru} and \eqref{glu}.
\end{proof}

\begin{lemma}
  \label{LA1}
  Assume that\/ $\dd$ satisfies  \eqref{D2lim} and \eqref{dmaxo}.
  Then
  \begin{align}\label{la1b}
\frac{\Xb}{n}&\pto \frac{\nu}{2},
\\ \label{la1c}  % &
          \frac{\Xc}{n}&\pto \frac{\nu^2}{2\mu}.
  \end{align}
\end{lemma}

\begin{proof}
First, the overcount in \eqref{fru} comes from the loops and multiple edges,
  and we have the
estimate
\begin{align}\label{gal}
  0\le\sumin\binom{d_i}2-\Xb
  \le
  2\dmax X_{\sC_1}+ 2 X_{\sC_2}
  =  2\dmax L+2M
.
\end{align}
By \eqref{EL}--\eqref{EM},
$L,M=\Op(1)$, and thus \eqref{gal} yields, using \eqref{dmaxo},
\begin{align}
  \frac{\Xb}{n}
  =\frac{1}{n}\sumin\binom{d_i}2 +\op\bigpar{n\qqw}
    =\frac{1}{n}\sumin\binom{d_i}2 +\op(1).
\end{align}
Hence, %$\Xb/n\pto\nu/2$
\eqref{la1b} follows by \eqref{nu}.

For $\Xc$, we consider again $\XXc$ defined in \eqref{XXc}.
This too is defined by overcounting, and it is easily seen that
\begin{align}\label{thr}
  0\le\XXc-\Xc
  \le   2\dmax^2 X_{\sC_1}+4\dmax X_{\sC_2} +  3 X_{\sC_3}.
\end{align}
Again, $X_{\sC_1}=L=\Op(1)$ and $X_{\sC_2}=M=\Op(1)$ by \eqref{EL}--\eqref{EM}, and
a similar calculation shows $\E X_{\sC_3}=O(1)$
and thus $X_{\sC_3}=\Op(1)$.
Furthermore, $\dmax=o(n\qq)$ by \eqref{dmaxo}.
Hence,  \eqref{thr} implies
\begin{equation}
  \label{xru}
  \XXc-\Xc=\op(n).
\end{equation}

Consequently, it suffices to consider $\XXc$.
We have, by \eqref{EXXc},
\begin{align}\label{ser}
  \E\XXc
  = \frac{1}{2(N-1)}\lrpar{\biggpar{\sumin d_i(d_i-1)}^2-
  \sumin d_i^2(d_i-1)^2}.
\end{align}
Furthermore, by \eqref{dmaxo} and \eqref{D2},
\begin{align}\label{cys}
    \sumin d_i^2(d_i-1)^2\le \dmax^2\sumin d_i^2 =o(n^2).
\end{align}
Using \eqref{cys} and \eqref{nu} in \eqref{ser} yields
\begin{align}\label{met}
 \frac{ \E\XXc}{n}\to\frac{\nu^2}{2\mu}.
\end{align}

Finally, we estimate the variance of $\XXc$.
We use again the representation \eqref{XXc}.
We have the covariances
\begin{align}
  \Cov\bigpar{I_{i,\ga,j,\gb},I_{k,\gam,\ell,\gd}}
  =
  \begin{cases}
    % \E I_{i,\ga,j,\gb} - \bigpar{\E I_{i,\ga,j,\gb}}^2,
    % (N-1)\qw-(N-1)\qww,
        \frac{1}{N-1}-\bigpar{\frac{1}{N-1}}^2,
    & \set{i_\ga,j_\gb}=\set{k_\gam,\ell_\gd},
    \\
    -\bigpar{\frac{1}{N-1}}^2,
    & |\set{i_\ga,j_\gb}\cap\set{k_\gam,\ell_\gd}|=1,
    \\
    \frac{1}{(N-1)(N-3)}-\bigpar{\frac{1}{N-1}}^2,
    & |\set{i_\ga,j_\gb}\cap\set{k_\gam,\ell_\gd}|=0.
  \end{cases}
\end{align}
Hence, whenever $\set{i_\ga,j_\gb}\neq\set{k_\gam,\ell_\gd}$,
\begin{align}
  \Cov\bigpar{I_{i,\ga,j,\gb},I_{k,\gam,\ell,\gd}}
  \le    \frac{2}{(N-1)^2(N-3)}
  \le \frac{C}{n} \E\Iiajb \E I_{k,\gam,\ell,\gd},
\end{align}
and it follows from \eqref{XXc} that, using \eqref{EXXc} and \eqref{dmaxo},
\begin{align}\label{XXc2}
  \Var\XXc
  &\le
    \sum_{i<j}\sum_{\ga=1}^{d_i}\sum_{\gb=1}^{d_j} (d_i-1)^2(d_j-1)^2\E\Iiajb
    +\frac{C}{n} \bigpar{\E\XXc}^2
    \notag\\&
  \le \dmax^2\E\XXc  +\frac{C}{n} \bigpar{\E\XXc}^2
  =O\bigpar{n\dmax^2}+O\bigpar{n}
  =o\bigpar{n^2}.
\end{align}
Consequently, $(\XXc-\E\XXc)/n\pto0$, which together with
\eqref{met} implies
\begin{equation}
  \label{xyl}
    \frac{\XXc}{n}\pto\frac{\nu^2}{2\mu}.
\end{equation}
Finally, this and \eqref{xru} imply \eqref{la1c}.
\end{proof}

\begin{remark}
  \refL{LA1} may fail without the assumption \eqref{dmaxo}.
  For an example, consider again \refE{EO}. Then
  $\XXc=(m-1)^2M_{12}$ and by an estimate similar to \eqref{thr},
  $\Xc=(m-1)^2 M_{12}+\op(n)$;
  hence $\Xc/n=aM_{12}+\op(1)\dto a\Po(a)$. Thus $\Xc/n$ does not converge
 in probability  to a constant.
\end{remark}

\begin{remark}
  Under the stronger assumption $\sum_i d_i^m = O(n)$ for every $m<\infty$,
  it is shown in \cite[Theorem 3.10]{SJ338} that (for example)
  $\Xb$ and $\Xc$ are asymptotically normal, with variance of order $n$.
  We do not know whether that holds under the weaker assumptions in \refL{LA1}.
\end{remark}

\begin{lemma}
  \label{LA2}
  Assume that\/ $\dd$ satisfies  \eqref{D2lim} and \eqref{dmaxo}.
  Then, for any $\ell,m\ge0$,
  \begin{align}\label{la2}
    n^{-\ell-m}\ZZ\lm &\pto
\ga\lm:=\frac{1}{\ell!\,m!}\Bigparfrac{\nu}{2}^\ell
\Bigpar{\frac{\nu^2}{2\mu}}^m.
  \end{align}
\end{lemma}

\begin{proof}
  $\ell!\,m!\,\ZZ\lm$ counts ordered sequences of subgraphs
  $F_1,\dots, F_\ell,\allowbreak H_1,\dots,H_m$ of vertex-disjoint
paths in $\ggndd$ 
such that each $F_i\cong\sPb$ and each $H_j\cong\sPc$,
and \ref{Pgap2} holds.
We may overcount and estimate this by $\Xb^\ell \Xc^m$;
we can also estimate the error by
\begin{multline}\label{trp}
  0 \le \Xb^\ell\Xc^m - \ell!\,m!\,\ZZ\lm
  \\
  \le \binom\ell2 A_{22}\Xb^{\ell-2}\Xc^{m}
  +\ell m  A_{23}\Xb^{\ell-1}\Xc^{m-1} + \binom{m}2 A_{33}\Xb^{\ell}\Xc^{m-2}
\\  + \ell B_2 \Xb^{\ell-1}\Xc^m   + m B_3 \Xb^{\ell}\Xc^{m-1},
\end{multline}
where $A_{jk}$ is that number of pair of paths $F\cong\sP_j$ and $F'\cong\sP_k$
in $\ggndd$ such that $F\cap F'\neq\emptyset$,
and $B_j$ is the number of paths $\sP_j$ such that \ref{Pgap2} does not
hold, \ie{} paths $\sP_j$ that are part of a cycle $\sC_{j+1}$.

We estimate $A_{jk}$ and $B_j$. First,
we have $B_2\le 3X_{\sC_3}$ and $B_3\le 4X_{\sC_4}$.
Calculations similar to \eqref{EM} show $\E X_{\sC_j}=O(1)$ for any fixed
$j$ (as said in the proof of \refL{LA1} for $j=3$),
and thus
\begin{align}\label{ala}
  B_2&=\Op(1),& B_3&=\Op(1).
\end{align}

Fix $j,k\in\set{2,3}$.
We make a decomposition
\begin{align}\label{val}
  A_{jk}=\sum_{H\in\cH} A_{jk}(H),
\end{align}
where
$A_{jk}(H)$ is the number of pairs $(F,F')$ of paths in $\ggndd$ such that
$F\cap F'\neq\emptyset$, $F\cong \sP_j$, $F'\cong\sP_k$ and $F\cup F'\cong H$,
and $\cH$ is the (finite) set of unlabelled multigraphs that can be written
as a union of two paths of lengths $j$ and $k$.
Given $F\cup F'$, there is $O(1)$ choices of $F$ and $F'$, and thus
\begin{align}\label{gly}
  A_{ij}(H) \le C X_H.
\end{align}

For $r\ge0$, let
\begin{align}\label{Dr}
\gD_r:=\sumin d_i^r.  
\end{align}
Let $H\in\cH$ have $q$ vertices with degrees $\gd_1,\dots,\gd_q$.
Then the number of possible copies of $H$ in $\ggndd$, counted in the
corresponding configuration, is at most
\begin{align}
  \sum_{i_1,\dots,i_q} \prod_{k=1}^q d_{i_k}^{\gd_k}
  =\prod_{k=1}^q \gD_{\gd_k},
\end{align}
and 
each such copy occurs with probability $O(N^{-e(H)})=O(n^{-e(H)})$.
Since  $e(H)=\frac12\sum_k \gd_k$, 
\begin{align}\label{tyr}
  \E X_H\le C n^{-e(H)}\prod_{k=1}^q \gD_{\gd_k}
  = C \prod_{k=1}^q n^{-\gd_k/2}\gD_{\gd_k}.
\end{align}
By \eqref{D2}, $\gD_1=O(n)$ and $\gD_2=O(n)$.
Furthermore, for $\gd>2$, $\gD_\gd\le\dmax^{\gd-2}\gD_2\le C\dmax^{\gd-2}n$.
Hence, using \eqref{dmaxo},
\begin{align}\label{phe}
  n^{-\gd/2}\gD_\gd
  \le
  \begin{cases}
    C n\qq, & \gd=1,\\
    C , & \gd=2,\\
            C\dmax^{\gd-2} n^{1-\gd/2}=o(1) , & \gd>2.
  \end{cases}
\end{align}

Let $h_1$ be the number of vertices in $H$ with degree 1.
Since $H$ is a connected union of two paths,
$H$ has no isolated vertices, and it follows from
\eqref{tyr} and \eqref{phe} that $\E X_H=O(n^{h_1/2})$.
Furthermore,
$h_1\le 4$, and if $h_1=4$, then
there is some vertex with degree $>2$.
If $h_1\le3$, then $\E X_H=O\bigpar{n^{3/2}}$, and if $h_1=4$ and
some vertex has degree $\gd>2$, then
\eqref{tyr} and \eqref{phe} imply
$\E X_H \le C n^{4/2}(\dmax/n\qq)^{\gd-2}=o\bigpar{n^2}$.

Hence, $\E X_H=o\bigpar{n^2}$ for every $H\in\cH$, and thus
$\E A_{jk}(H)=o\bigpar{n^2}$ by \eqref{gly}, and finally
$\E A_{jk}=o\bigpar{n^2}$ by \eqref{val},
which implies
\begin{align}
  \label{ile}
A_{jk}=\op\bigpar{n^2}.
\end{align}

Since \refL{LA1} implies $\Xb,\Xc\le Cn$ \whp, 
it follows from \eqref{ala} that the last two terms in \eqref{trp} are
$\Op\bigpar{n^{\ell+m-1}}$, and from \eqref{ile} that the remaining terms on the
\rhs{} are $\op\bigpar{n^{\ell+m}}$.
Consequently, \eqref{trp} implies
\begin{align}\label{trpq}
 \Xb^\ell\Xc^m - \ell!\,m!\,\ZZ\lm
 =\op\bigpar{n^{\ell+m}}
\end{align}
and thus
\begin{align}\label{trpr}
  n^{-\ell-m}\ZZ\lm
  =\frac{1}{\ell!\,m!}\parfrac{\Xb}{n}^\ell\parfrac{\Xc}{n}^m +\op(1)
\end{align}
and the result \eqref{la2} follows from \refL{LA1}.
\end{proof}

\begin{lemma}
  \label{LA3}
  Assume that\/ $\dd$ satisfies  \eqref{D2lim} and \eqref{dmaxo}.
  Let $\ell,m\ge0$. Then,
  \begin{align}
    \label{la3}
        \E|\ZZ\lm-\ga\lm n^{\ell+m}|    =o\bigpar{n^{\ell+m}}.
  \end{align}
%\begin{gather}
%  \label{la3a}
%  \E\ZZ\lm= \bigpar{\ga\lm+o(1)}n^{\ell+m},
%  \\\label{la3b}
%    \E|\ZZ\lm-\E\ZZ\lm|    =o\bigpar{n^{\ell+m}}.
%  \end{gather}
\end{lemma}

\begin{proof}
  Let $Y_n:=  n^{-\ell-m}\ZZ\lm$.
We have $Y_n\le n^{-\ell-m}\Xb^\ell\Xc^m$ by \eqref{trp}, and thus
by \refL{LA0},
\begin{align}
  % \E\bigpar{\abs{n^{-\ell-m}\ZZ\lm}^2}
  \E [Y_n^2]
  \le n^{-2\ell-2m}\E\bigsqpar{\Xb^{2\ell}\Xc^{2m}}=O(1).
\end{align}
Hence, the sequence
% $n^{-\ell-m}\ZZ\lm$
$Y_n$ ($n\ge1$) is uniformly integrable
(see \eg{} \cite[Theorem 5.4.2]{Gut}), and thus \eqref{la2} implies
$L^1$-convergence \cite[Theorem 5.5.4]{Gut}, \ie,
\begin{align}\label{rib}
  % \E\bigabs{n^{-\ell-m}\ZZ\lm-\ga\lm}\to0.
    \E\bigabs{Y_n-\ga\lm}\to0,
\end{align}
which is equivalent to \eqref{la3}.
\end{proof}

\section{Proof of \refT{T1}}\label{SpfT1}

%\section{Subgraph counts in $\gndd$}\label{Sg}

We now transfer the results in \refS{Sgg} to the simple random graph $\gndd$.

  \begin{lemma}\label{LB}
    The results in Lemmas \ref{LA0}, \ref{LA1}, \ref{LA2}, \ref{LA3}
    hold also conditioned on the event that $\ggndd$ is simple.
    In other words, the corresponding results for $\gndd$ hold as well.
  \end{lemma}
  \begin{proof}
    An immediate consequence of \eqref{liminf}.
  \end{proof}

  Let
  $Z\lm:=\zeta\lm(\gndd)$.

  \begin{lemma}
  \label{LB3}
  Assume that\/ $\dd$ satisfies  \eqref{D2lim} and \eqref{dmaxo}.
  Let $\ell,m\ge0$. Then,
\begin{gather}
  \label{lb3a}
  \E Z\lm= \bigpar{\ga\lm+o(1)}n^{\ell+m},
  \\\label{lb3b}
    \E| Z\lm-\E Z\lm|    =o\bigpar{n^{\ell+m}}.
  \end{gather}
\end{lemma}
\begin{proof}
  By \refLs{LA3} and \ref{LB},
  \begin{align}
    \label{lbb3}
%    \bigabs{\E Z\lm-\ga\lm n^{\ell+m}} \le
            \E\bigabs{Z\lm-\ga\lm n^{\ell+m}}    =o\bigpar{n^{\ell+m}}.
  \end{align}
This implies \eqref{lb3a}, and then \eqref{lb3a} and \eqref{lbb3} yield
\eqref{lb3b}.
\end{proof}

\begin{proof}[Proof of \refT{T1}]
  As said in \refS{Ssubsub}, we may assume that \eqref{D2lim} holds and $\nu>0$.
  Fix $s\ge0$.
  Recalling the notations $\Zgs:=\zetags\bigpar{\gndd}$ and
  $Z\lm:=\zeta\lm\bigpar{\gndd}$, we see that \eqref{alice} and \eqref{lb3b}
  imply that %for every fixed $s\ge0$, 
  \begin{align}
    \label{lbb4}
        \E|\Zgs-\E\Zgs| =o\bigpar{n^{s}}.
  \end{align}
%  where $\gb_s:=\sum_{\ell+m=s}2^{-m}\ga\lm>0$.
  Furthermore, \eqref{alice} and \eqref{lb3a}  imply
  \begin{equation}\label{lbyx}
    \E\Zgs \ge \E Z_{s,0} =\bigpar{\ga_{s,0}+o(1)}n^s.
  \end{equation}
  Hence, using \eqref{bob},
  noting that $\ga_{s,0}>0$, % when $\nu>0$,
  \begin{align}
    \label{lbbb}
      \normm{\hmugs-\umu}
=        \frac{\E|\Zgs-\E\Zgs|}{\E\Zgs}\to0.
  \end{align}
  
Let $p_s:=\P(\cG_s)$, and $\pc:=\P(\bcG)$. Then $\pc+\sum_{s}p_s=1$, and,
recalling \eqref{murg} and  letting $\hmugc$ be the distribution of
$\hgndd$ conditioned on $\bcG$,
\begin{align}
  \hmu = \pc\hmugc+\sum_{s=0}^\infty p_s\hmugs.
\end{align}
Consequently, 
for any $K\ge1$,
using \eqref{bob},
\begin{align}\label{slu}
  \normm{\hmu-\umu}&
  =\Bignorm{\pc(\hmugc-\umu)+\sum_{s\ge0} p_s(\hmugs-\umu)}_{\cM(\GGn)}
  \notag\\&                   
  \le \pc\normm{\hmugc-\umu}+\sum_{s\ge 0} p_s \normm{\hmugs-\umu}
                    \notag\\&
% \le  2\pc+\sum_{s=0}^K p_s\frac{\E\bigabs{\Zgs-\E \Zgs}}{\E \Zgs}
    \le  2\pc+\sum_{s=0}^K\normm{\hmugs-\umu}  +\sum_{s>K}2 p_s.
\end{align}
By \refL{LSG}\ref{LSG:G}, $\pc=o(1)$, and by \eqref{lbbb},
  $\normm{\hmugs-\umu}=o(1)$ for every fixed $s$.
  Furthermore, 
    recall that $\cG_s$ implies $L+M=S=s$, see \eqref{SG}.
  Hence, for any fixed $K$, \eqref{slu} implies, using \eqref{ELMC},
  \begin{align}\label{erika}
  \normm{\hmu-\umu}&
\le  o(1)+o(1) %+\sum_{r+s\le K}\frac{\E\bigabs{Z_{r,s}-\E Z_{t,s}}}{\E Z_{r,s}}
                    +\sum_{s>K}2\P(L+M=s)
%                    \notag\\&
    = o(1)+2\P(L+M>K)
    \notag\\&
    \le o(1)+\frac{2\E(L+M)}{K}
                            \le o(1)+\frac{C}{K}.
  \end{align}
  Thus $\limsup_\ntoo\normm{\hmu-\umu}\le C/K$,
  and then letting $K\to\infty$ yields
  $\normm{\hmu-\umu}\to0$.
  This completes the
proof of \eqref{t1a} by \eqref{dtv1}. The final sentence follows by
\eqref{dtv3}.
\end{proof}

\begin{proof}[Proof of \refC{C1}]
  Using a coupling such that \eqref{t1b} holds, we have $\P(X_n\neq\hX_n)\to0$,
and the conclusion follows.
\end{proof}

\begin{proof}[Proof of \refC{C2}]
  Let $\hX_n:=f_n\bigpar{\hgndd}$.
  Then the assumption \eqref{diff0} says $\hX_n-\xX_n\pto0$.
  We have also assumed $\xX_n\dto Y$, and it follows that $\hX\dto Y$.
  Hence the result follows by \refC{C1}.
\end{proof}

\section{Proof of \refT{T2}}\label{SpfT2}

Since we do not assume \eqref{dmaxo},
the construction is not necessarily golden \whpx.
(There may be \eg{} triple edges in $\ggndd$.)
Hence we use in this section silver constructions.
Recall the notation $\Zss:=\zetass(\gndd)$.

We define, \cf{} \eqref{trox}, for a multigraph $H$,
\begin{align}\label{troy}
  Y_H:=\psi_H\bigpar{\gndd}.
\end{align}

\begin{lemma}
  \label{LBB}
  Assume that\/ $\dd$ satisfies  \eqref{D2lim} and $\nu>0$.
  Let $s\ge0$.
  Then, 
  \begin{align}
    \E\Zss &\ge c n^s,\label{lbb1}
    \\
       \E\Zss^2 &\le C n^{2s},\label{lbb2}
  \end{align}
\end{lemma}

\begin{proof}
  \pfitem{i}
  By selecting subsequences, see \refS{Ssubsub}, we may assume that the
  limit $b:=\lim_\ntoo \dmax/n\qq$ exists.
  We consider two cases.
  \pfitem{ia} $b=0$. This means $\dmax=o(n\qq)$, \ie, \eqref{dmaxo} holds;
  thus \refL{LB3} and its consequence \eqref{lbyx} hold.
  By definition, $\zetass(G)\ge\zetags(G)$ for all $G\in\GGn$.
Hence, $\Zss\ge\Zgs$, and \eqref{lbb1} follows by \eqref{lbyx}.

\pfitem{ib}
$b>0$. We may assume that $\dmax=d_1$. Then, for large $n$,
$d_1=\dmax\ge \frac{b}{2}n$. Say that a half-edge at vertex 1 is
\emph{green} if it is not part of a loop or multiple edge, and let $W$ be
the number of green edges.
Then $W\ge d_1-2L-2M\ge \frac{b}{2} n\qq -\Op(1)$, and thus \whp{} $W\ge
\frac{b}{3}n\qq$.
A pair of green half-edges defines a path of length 2 in $\gndd$, with 1 as
midpoint.
Hence, a sequence of $2s$ distinct green half-edges defines $s$ paths of 
length 2. This set of paths satisfies \ref{P23}--\ref{Pgap1}
and (trivially) \ref{Plast}. It fails
to satisfy \ref{Pgap2} only if one of the paths is part of a $\sC_3$,
and the number of such sequences is $O(Y_{\sC_3}W^{2s-2})$.
%, where
%$Y_{\sC_3}:=\psi_{\sC_3}(\gndd)$ is the number of subgraphs $\sC_3$.
We have, by the usual conditioning argument with \eqref{liminf},
and a simple calculation
(used also in the proof of \refL{LA1}),
\begin{align}
  \E Y_{\sC_3}
  =
  \E \bigpar{X_{\sC_3}\mid\ggndd\text{ is simple}}
  \le C\E X_{\sC_3} =O(1).
\end{align}
Hence, \whp{} $Y_{\sC_3}\le n\qq$.
Consequently, \whp, crudely,
\begin{align}
  \Zss\ge \binom{W}{2s} - CW^{2s-2}Y_{\sC_3}
  \ge c b^{2s} n^s -  O\bigpar{n^{s-1/2}}
    \ge c n^s,
\end{align}
which implies \eqref{lbb1}.

  \pfitem{ii}
%Let $Y_k:=\psi_{\cP_k}(\gndd)$, the number of paths of length $k$ in $\gndd$.
  The weights $w(\cR;G)\in\oi$, and thus
  $\zetass(G)$ is at most the number of sets of $s$ paths of
  lengths 2 or 3 in $G$. Hence,
  \begin{align}
    \Zss=
    \zetass(\gndd) \le \sum_{\ell=0}^s \Yb^\ell \Yc^{s-\ell}.
  \end{align}
  Thus, for fixed $s\ge0$,
  recalling $X_k=X_{\sP_k}$ in \refS{Sgg} and \refL{LA0}, and \eqref{liminf},
  \begin{align}
    \E \Zss^2
    &\le C\sum_{\ell=0}^{2s} \E[\Yb^\ell \Yc^{2s-\ell}]
 =C\sum_{\ell=0}^{2s} \E\bigpar{X_2^\ell X_3^{2s-\ell}\mid \ggndd \text{ is simple}}
      \notag\\&
    \le 
    \frac{C}{\P\bigpar{\ggndd \text{ is simple}}}
    \sum_{\ell=0}^{2s} \E\bigpar{X_2^\ell X_3^{2s-\ell}}
    =O\bigpar{n^{2s}}.
  \end{align}
\end{proof}

\begin{proof}[Proof of \refT{T2}]
Again, as said in \refS{Ssubsub}, we may assume that \eqref{D2lim} holds and
$\nu>0$. 
  
Recall the definition \eqref{contig}, and
  fix an arbitrary sequence of subsets $\cE_n\subseteq\GGn$.
(An event for these random graphs may by identified with a subset
of $\GGn$.)

  \pfitem{i}
  Suppose that
  \begin{align}\label{hlim0}
  \P\bigpar{\hgndd\in\cE_n}\to0.  
  \end{align}
  The event that $\ggndd$ is simple is the same as $S=0$
  (\ie, no switchings are
    made), and in this case $\hgndd=\ggndd$. Hence,
    \begin{align}
      \P\bigpar{\gndd\in\cE_n}
      &=
\P\bigpar{\ggndd\in\cE_n\mid S=0}
        =
        \frac{  \P\bigpar{\ggndd\in\cE_n \text{ and } S=0} }{\P(S=0)}
        \notag\\&
            =
        \frac{  \P\bigpar{\hgndd\in\cE_n \text{ and } S=0} }{\P(S=0)}
      \le
              \frac{  \P\bigpar{\hgndd\in\cE_n }}{\P(S=0)}
\to0,
    \end{align}
by the assumption \eqref{hlim0} and \eqref{liminf}.    

  \pfitem{ii}
  Suppose conversely that
  \begin{align}\label{lim0}
    %\umu(\cE_n)=
    \P\bigpar{\gndd\in\cE_n}\to0.  
  \end{align}
  Fix $s\ge0$ and let now $p_s:=\P(\cS_s)$.
  Then \eqref{murs} yields
  \begin{align}\label{tok}
    \P\bigpar{\hgndd\in\cE_n \text{ and }\cS_s}
    =
    p_s    \P\bigpar{\hgndd\in\cE_n \mid\cS_s}
    =p_s\hmuss(\cE_n).
  \end{align}
Furthermore, \eqref{bw} implies
\begin{align}\label{tqk}
  \hmuss(\cE_n)
=\sum_{G\in\cE_n}\frac{\zetass(G)}{\E\Zss}\gl\set{G}
  =\frac{ \E [\Zss\indic{\gndd\in\cE_n}]}{\E\Zss}.
\end{align}
The \CSineq{} and \eqref{tqk} yield, using \eqref{lbb1}--\eqref{lbb2} and
the assumption \eqref{lim0},
\begin{align}\label{qtk}
  \hmuss(\cE_n)
\le\frac{\bigpar{ \E [\Zss^2] \P\xpar{\gndd\in\cE_n}}\qq}{\E\Zss}
  \le C \P\xpar{\gndd\in\cE_n}\qq
  \to0.
\end{align}
Hence, for every fixed $s\ge0$, \eqref{tok} and \eqref{qtk} imply
\begin{align}
  \label{rtk}
  \P\bigpar{\hgndd\in\cE_n \text{ and }\cS_s}\to0.
\end{align}

We now argue similarly to the final part of the proof of \refT{T1}.
For every fixed $K\ge1$,
using also \refL{LSG}\ref{LSG:S},
\eqref{SS2} and \eqref{ELMC},
\begin{align}
  \P\bigpar{\hgndd\in\cE_n}
&  \le \P(\bcS)+
  \sum_{s=0}^\infty \P\bigpar{\hgndd\in\cE_n \text{ and }\cS_s}
  \notag\\&
  \le \P(\bcS)+
  \sum_{s=0}^K \P\bigpar{\hgndd\in\cE_n \text{ and }\cS_s}
  +\sum_{s>K}\P\bigpar{\cS_s}
  \notag\\&
  = o(1)+o(1) + \P(\cS\text{ and }S>K).
  \notag\\&
  \le o(1) + \P(L+M>K)
    \le o(1) + \frac{\E(L+M)}{K}
  \notag\\&
    \le o(1) + \frac{C}{K}.
\end{align}
Consequently,
$\limsup_\ntoo   \P\bigpar{\hgndd\in\cE_n}\le C/K$, and then letting
$K\to\infty$ yields \eqref{hlim0}.
\end{proof}

\section{Applications}\label{Sapp}

Let the random variable $D_n$ be the degree of a uniformly random vertex,
and note that \eqref{D2lim} can be written $\E D_n\to\mu$ and $\E
D_n^2\to\mu_2$.
We will in
the applications below use the standard assumption that there exists a
random variable $D$ such that
\begin{align}
  \label{Dlim}
  D_n\dto D.
\end{align}
%equivalently, there exists a probability distribution $(p_k)\xoo$ such that
%for each fixed $k$
%\begin{align}
%  \P(D_n=k)\to p_k.
%\end{align}
We will also sometimes assume that \eqref{D2lim} is strengthened to
\begin{align}
  \label{D2ui}
  \E D_n^2\to\E D^2<\infty.
\end{align}
Equivalently, assuming \eqref{Dlim},
the  sequence $D_n^2$ is uniformly integrable, see \cite[Theorem 5.5.9]{Gut}.
Note that this implies that \eqref{D2lim} holds with $\mu=\E D$.
It is also easy to see that
\eqref{Dlim}--\eqref{D2ui} imply
\eqref{dmaxo}.

\begin{example}\label{Egiant}
  Assume \eqref{Dlim}--\eqref{D2ui} and $\P(D=1)>0$.
  Assume also that $\nu-\mu=\E D(D-2)>0$; this is the \emph{supercritical}
  case where there is \whp{} a giant component of order $\Theta(n)$ in both
  $\ggndd$ and $\gndd$,
  see  \citet{MolloyReed95,MolloyReed98} with refinements in, \eg,
\cite{SJ204}, 
\cite{BollobasRiordan-old},
\cite{JossEtAl}.

Let $|\cC_k|=|\cC_k(G)|$ be the order of the $k$th largest component in a
  multigraph $G$. 
  It was proved by
  \citet{BarbourR} (under somewhat stronger assumptions),
  with a different proof in \cite{SJ338} (under the conditions here),
  that the size $|\cC_1|$ of the giant component is asymptotically normal
  for $\ggndd$:
  \begin{align}\label{giant}
    \frac{|\cC_1(\ggndd)|-\E|\cC_1(\ggndd)|}{\sqrt{n}}\dto N(0,\gss),
  \end{align}
  where the asymptotic variance $\gss$ was calculated explicitly by
  \citet{BallNeal}.

  It is shown in \cite{SJ338}, by a non-trivial extra argument,
  that \eqref{giant} holds also for $\gndd$.
  We can now replace that argument, and give a simpler proof of asymptotic
  normality for $\gndd$.

Consider the components of a multigraph as sets of vertices (ignoring the
edges).
A switching will either leave all components unchanged, or it will merge two
components. Hence, a sequence of $S$ switchings will change the size
$|\cC_1|$ of the largest component by at most $S|\cC_2|$; consequently,
\begin{align}\label{c1c2}
\bigabs{ |\cC_1(\hgndd)|-|\cC_1(\ggndd)| }\le S|\cC_2(\ggndd)|.
\end{align}
Furthermore, under our assumptions, 
\cite[Theorem 2]{BollobasRiordan-old}  and \cite[Lemma 9.4]{SJ338}
imply $|\cC_2(\ggndd)|\le C \log n$ \whp,
while \refT{Tnotbad} yields $S=\Op(1)$.
Hence, 
$S|\cC_2(\ggndd)|=\op(n\qq)$, 
and \eqref{c1c2} shows that \eqref{diff0} holds for
$f_n(G):=n\qqw\bigpar{|\cC_1(G)|-\E|\cC_1(\ggndd)|}$.
Consequently, \refC{C2} applies and shows
that \eqref{giant} implies
  \begin{align}\label{giantG}
    \frac{|\cC_1(\gndd)|-\E|\cC_1(\ggndd)|}{\sqrt{n}}\dto N(0,\gss),
  \end{align}
  Furthermore, if $X_n$ denotes the \lhs{} of \eqref{giant}, then also
  $\E X_n^2\to\gss$ \cite{BarbourR,SJ338}, and thus $X_n^2$ are uniformly
  integrable \cite[Theorem 5.5.9]{Gut}. Hence, using \eqref{liminf},
  $X_n^2$ are uniformly integrable also
  conditioned on $\ggndd$ being simple,
  and thus the mean and variance converge in \eqref{giantG} too.
  In particular,
  $\E|\cC_1(\gndd)|-\E|\cC_1(\ggndd)|=o\xpar{\sqrt{n}}$,
  and thus $\E|\cC_1(\ggndd)|$ can be replaced by $\E|\cC_1(\gndd)|$ in
  \eqref{giantG}. 
\end{example}

\begin{remark}
  \citet{Ball} has proved related results on asymptotic normality for the
  size of SIR epidemics on $\ggndd$.
  As a special case, he obtains asymptotic normality of the size of the
  giant component 
  for (bond or site) percolation in $\ggndd$ (in the supercritical case).
%  (Such results for percolation on   $\ggndd$ can also be deduced from
%  \eqref{giant} by the technique of \cite{SJ215}, showing the equivalence of
%  the percolated model with a configuration model with random degree sequences,
%see also \cite[Section 12]{SJ338}; we omit the details.)

  It seems possible that at least some of these
  results too can be  extended to $\gndd$ by our \refC{C2}, but it remains
  to verify \eqref{diff0} for them.
\end{remark}

\begin{example}\label{Etrees}
Let $T$ be a fixed tree, and let $n_T(G)$ be the number of components
isomorphic to $T$ in a (multi)graph $G$.
  Assume \eqref{Dlim}--\eqref{D2ui}. % and $\P(D=1)>0$.
  Then, by
    \citet{BarbourR} (under somewhat stronger assumptions),
and \cite{SJ338} (with a different proof)
    \begin{align}\label{nT}
    \frac{n_T(\ggndd)-\E n_T(\ggndd)}{\sqrt{n}}\dto N(0,\gss_T),
  \end{align}
  for some $\gss_T\ge0$ (with $\gss_T>0$ except in some rather trivial cases).
  It was shown in \cite{SJ338} that \eqref{nT} holds also for $\gndd$, again
  with an extra argument;
  we can now replace that by a simpler proof.
  
  A switching can change at most two components, and thus $n_T$ is changed
  by at most 2; hence, using \refT{Tnotbad},
\begin{align}%\label{c1c2}
\bigabs{n_T(\hgndd)-n_T(\ggndd) }\le 2S = \Op(1).
\end{align}
Consequently, \refC{C2} applies as in \refE{Egiant}, and shows that 
\eqref{nT} holds also for $\gndd$.
\end{example}

\begin{example}\label{Esj313}
  Assume 
  \eqref{Dlim}, $\P\bigpar{D\notin\set{0,2}>0}$
  and \eqref{D2}.
  Assume also
  \begin{align}
    \eps_n:=\frac{\E D_n(D_n-2)}{\E D_n}
=O\bigpar{n^{-1/3}(\E D_n^3)^{2/3}},
%    \to0,
  \end{align}
  which means that we are in the \emph{critical window}, and that
  \begin{align}
    \dmax
     =o\bigpar{n^{1/3}(\E D_n^3)^{1/3}},
  \end{align}
  which easily is seen to imply \eqref{dmaxo}.
  Then,
  \citet[Theorem 1.1]{HatamiMolloy} (under somewhat stronger conditions)
  showed that $|\cC_1|$ is of the order $\Upsilon_ n:=n^{2/3}(\E D_n^3)^{-1/3}$.
  Moreover, see \citet[Theorem 2.12]{SJ313},
  $|\cC_1|/\Upsilon_n$ is bounded in probability, but not bounded by a
  constant \whp;
  in other words:
  \begin{romenumerate}
  \item
    For any $\gd>0$ there exists $K=K(\gd)$ such that
    \begin{align}\label{bari}
    \P(|\cC_1|>K\Upsilon_n)\le \gd.  
    \end{align}
    
  \item
    For any $K<\infty$,
    \begin{align}
      \label{barii}
      \liminf_\ntoo \P(|\cC_1|>K\Upsilon_n)>0.
          \end{align}
  \end{romenumerate}
Both parts hold for both $\ggndd$ and $\gndd$; however, there is a technical
difference in the proofs.
Part (i) is proved, by both \cite{HatamiMolloy} and \cite{SJ313} (with
different methods) first for $\ggndd$, and the result for $\gndd$ then
follows immediately by the standard conditioning argument and \eqref{liminf}.

Part (ii) is also proved (by \cite{SJ313}) first for $\ggndd$, but here we
cannot use conditioning directly, and a rather long extra argument is
given in \cite[Section~6.3]{SJ313}.
We can now replace this extra argument by  \refT{T2}.

Note first that, as said in \refE{Egiant}, switchings can only merge
components, but never break them; hence, switchings can only increase
$|\cC_1|$, and thus if \eqref{barii} holds for $\ggndd$, then it holds for
$\hgndd$ too.
Suppose that \eqref{barii} fails for $\gndd$. Then there exists a
subsequence where the probability tends to 0, and the contiguity in
\refT{T2} shows that the same holds for $\hgndd$; a contradiction.
%  \cite[Theorem 2.12(ii)]{SJ313}
\end{example}

\section*{Acknowledgement}
I thank  Xing Shi Cai for making the pictures.

%This work  was partially supported by 
%a grant from the Knut and Alice Wallenberg Foundation.

\newcommand\AAP{\emph{Adv. Appl. Probab.} }
\newcommand\JAP{\emph{J. Appl. Probab.} }
\newcommand\JAMS{\emph{J. \AMS} }
\newcommand\MAMS{\emph{Memoirs \AMS} }
\newcommand\PAMS{\emph{Proc. \AMS} }
\newcommand\TAMS{\emph{Trans. \AMS} }
\newcommand\AnnMS{\emph{Ann. Math. Statist.} }
\newcommand\AnnPr{\emph{Ann. Probab.} }
\newcommand\CPC{\emph{Combin. Probab. Comput.} }
\newcommand\JMAA{\emph{J. Math. Anal. Appl.} }
\newcommand\RSA{\emph{Random Structures Algorithms} }
\newcommand\ZW{\emph{Z. Wahrsch. Verw. Gebiete} }
\newcommand\DMTCS{\jour{Discr. Math. Theor. Comput. Sci.} }

\newcommand\AMS{Amer. Math. Soc.}
\newcommand\Springer{Springer-Verlag}
\newcommand\Wiley{Wiley}

\newcommand\vol{\textbf}
\newcommand\jour{\emph}
\newcommand\book{\emph}
\newcommand\inbook{\emph}
\def\no#1#2,{\unskip#2, no. #1,} %(typeset after year) 
\newcommand\toappear{\unskip, to appear}

\newcommand\arxiv[1]{\texttt{arXiv:#1}}
\newcommand\arXiv{\arxiv}

\end{document}